\documentclass[reqno,12pt]{amsart}

\usepackage{amsmath,amsthm,amssymb,amsfonts,amscd,comment}
\usepackage{xypic}

\theoremstyle{plain} 
\newtheorem{thm}{Theorem}[section]
\newtheorem*{thm*}{Theorem}
\newtheorem{cor}[thm]{Corollary}
\newtheorem{lem}[thm]{Lemma}
\newtheorem{prop}[thm]{Proposition}
\newtheorem{conj}[thm]{Conjecture}
\newtheorem*{conj*}{Conjecture}

\newtheorem{dfn}[thm]{Definition}
	
\theoremstyle{definition}
\newtheorem{eg}[thm]{Example}

\newtheorem{rem}[thm]{Remark}	
\theoremstyle{remark}

\numberwithin{equation}{section}

\def\ZZ{{\mathbb Z}}
\def\QQ{{\mathbb Q}}
\def\RR{{\mathbb R}}
\def\CC{{\mathbb C}}

\def\OO{{\mathcal O}}

\def\PP{{\mathbb P}}

\def\A{{\mathcal A}}
\def\B{{\mathcal B}}
\def\C{{\mathcal C}}
\def\D{{\mathcal D}}
\def\E{{\mathcal E}}
\def\F{{\mathcal F}}

\def\M{{\mathcal M}}

\def\O{{\mathcal O}}
\def\P{{\mathcal P}}

\def\T{{\mathcal T}}
\def\U{{\mathcal U}}

\newcommand{\ch}{\mathrm{ch}}
\newcommand{\iso}{\xrightarrow{\sim}}
\newcommand{\chb}{\mathrm{ch}^{\beta}_{\mathcal{B}^l_0}}
\newcommand{\chbo}{\mathrm{ch}^{\beta}_{\mathcal{B}^l_0,0}}
\newcommand{\chbi}{\mathrm{ch}^{\beta}_{{\mathcal{B}^l_0},1}}
\newcommand{\chbii}{\mathrm{ch}^{\beta}_{{\mathcal{B}^l_0},2}}
\newcommand{\chbiii}{\mathrm{ch}^{\beta}_{{\mathcal{B}^l_0},3}}
\newcommand{\Aut}{\mathrm{Aut}}
\newcommand{\FM}{\mathrm{FM}}
\newcommand{\Stab}{\mathrm{Stab}}

\newcommand{\Hom}{{\rm Hom}}

\def \mf#1#2#3#4{
\xymatrix{
{#1}\  \ar@<0.4ex>[r]^{{#2}} & \ {#4}
\ar@<0.4ex>[l]^{{#3}}
}
}

\def \mfs#1#2#3#4{\!
\xymatrix@C=1,5em{{#1} \! \ar@<0.2ex>[r]^{{#2}} & \! {#4}
\ar@<0.2ex>[l]^{{#3}}
}
\!}

\def \mfl#1#2#3#4{
\xymatrix@C=2.6em{{#1}\  \ar@<0.4ex>[r]^{{#2}} &\  {#4}
\ar@<0.2ex>[l]^{{#3}}
}
}

\def \mfss#1#2#3#4{\!
\xymatrix@C=1.5em{{#1} \ar@<0.3ex>[r]^{{#2}} & {#4}
\ar@<0.3ex>[l]^{{#3}}
}
\!}
\begin{document}
\title{Automorphism groups of cubic fourfolds and K3 categories}
\author{Genki Ouchi}
\address{Interdisciplinary Theoretical and Mathematical Sciences Program, RIKEN, 2-1 Hirosawa, Wako, Saitama, 351-0198, Japan}
\email{genki.ouchi@riken.jp}
\begin{abstract}
In this paper, we study relations between automorphism groups of cubic fourfolds and Kuznetsov components. Firstly, we characterize automorphism groups of cubic fourfolds as subgroups of autoequivalence groups of Kuznetsov components using Bridgeland stability conditions. Secondly, we compare automorphism groups of cubic fourfolds with automorphism groups of their associated K3 surfaces. Thirdly, we note that the existence of a non-trivial symplectic automorphism on a cubic fourfold is related to the existence of associated K3 surfaces.
\end{abstract}
\maketitle

\section{Introduction}
\subsection{Background and results}
In this paper, we study relations between symmetries of cubic fourfolds and symmetries of K3 surfaces. Finite symmetries of K3 surfaces are related to sporadic finite simple groups as the Mathieu groups $M_{23}, M_{24}$ and the Conway groups $\mathrm{Co}_0, \mathrm{Co}_1$ in both mathematics and physics. Mukai \cite{Muk} proved that finite groups of symplectic automorphisms of K3 surfaces are certain subgroups of the Mathieu group $M_{23}$. In the context of physics, Eguchi, Ooguri and Tachikawa \cite{EOT} found Mathieu moonshine phenomena for the elliptic genera of K3 surfaces. This is the mysterious relation between the elliptic genera of K3 surfaces and the Mathieu group $M_{24}$. After \cite{EOT}, Gaberdiel, Hohenegger and Volpato \cite{GHV} studied symmetries of K3 sigma models. From the mathematical point of view, Gaberdiel, Hohenegger and Volpato \cite{GHV} compared symmetries of the Mukai lattice with the symmetries of the Leech lattice.  Huybrechts \cite{Huy} interpreted the symmetries of the Mukai lattice in \cite{GHV} as the symmetries of derived categories of K3 surfaces and stability conditions.
We recall the precise statement of Huybrechts's theorem in \cite{Huy}. For a K3 surface $S$, let $\Stab^*(S)$ be the distinguished connected component of the space of stability conditions on the derived category $D^b(S)$ of coherent sheaves on $S$ as in \cite{Bri08}.  For a K3 surface $S$ and a stability condition $\sigma \in \Stab^*(S)$, denote the group of symplectic autoequivalences of $D^b(S)$ fixing $\sigma$ by $\Aut_{\mathrm{s}}(D^b(S),\sigma)$ (Definition \ref{stabilizer K3}). The Conway group $\mathrm{Co}_0$ is the automorphism group of the Leech lattice $N$. The Conway group $\mathrm{Co}_1=\mathrm{Co}_0/\{\pm1\}$ is known as one of sporadic finite simple groups. 
Huybrechts \cite{Huy} proved the analogue of Mukai's theorem in \cite{Muk}.
\newpage
\begin{thm}[\cite{Huy}]\label{Conway}
Let $G$ be a group.
The following are equivalent.
\begin{itemize}
\item[(1)]There is a K3 surface $S$ and a stability condition $\sigma \in \Stab^*(S)$ such that $G$ can be embedded into $\Aut_{\mathrm{s}}(D^b(S),\sigma)$.
\item[(2)]The group $G$ can be embedded into the Conway group $\mathrm{Co}_0=\mathrm{O}(N)$ such that $\mathrm{rk}(N^G) \geq 4$. 
\end{itemize}
\end{thm}
 Gaberdiel, Hohenegger and Volpato \cite{GHV} computed the automorphism group of the Gepner model $(1)^6$. The automorphism group of the Gepner model $(1)^6$ is same as the symplectic automorphism group of the Fermat cubic fourfold. Cheng, Ferrari, Harrison and Paquette \cite{CFHP} studied Landau-Ginzburg models related to homogeneous polynomials of degree three with six variables. In \cite{CFHP}, symplectic automorphism groups of cubic fourfolds appeared as the symmetries of K3 sigma models. Recently, Laza and Zheng \cite{Laza} classified symplectic automorphism groups of cubic fourfolds.  In particular, symplectic automorphism groups of cubic fourfolds are certain subgroups of the Conway group $\mathrm{Co}_0$. The first main result of our paper is that the interpretation of automorphism groups of cubic fourfolds as certain subgroups of autoequivalence groups of Kuznetsov components of cubic fourfolds (cf. Theorem \ref{Conway}). Kuznetsov \cite{Kuz} studied semi-orthogonal decompositions of derived categories of cubic fourfolds.  For a cubic fourfold $X$, there is a semi-orthogonal decomposition 
\[D^b(X)=\langle \D_X, \OO_X, \OO_X(1), \OO_X(2) \rangle.\]
The admissible subcategory $\D_X$ of the derived category $D^b(X)$ of $X$ is called the Kuznetsov component of $X$. The kuznetsov component $\D_X$ of $X$ is a two dimensional Calabi-Yau category \cite{Kuz}.
By Orlov's theorem \cite{Orl}, the Kuznetsov component of a cubic fourfold $X$ is equivalent to the category of graded matrix factorizations of the defining polynomial of $X$.
Kuznetsov components of cubic fourfolds share various properties with derived categories of K3 surfaces. Addington and Thomas \cite{AT} introduced weight two Hodge structures on Mukai lattices for Kuznetsov components of cubic fourfolds. Then we have the notion of symplectic autoequivalences of  Kuznetsov components of cubic fourfolds. Bayer, Lahoz, Macr\`{i} and Stellari \cite{BLMS} constructed stability conditions on Kuznetsov components of cubic fourfolds.  We give the statement of the second result in our paper. Let $X$ be a cubic fourfold. Denote the stability condition on $\D_X$ constructed in \cite{BLMS} by $\sigma$ (Theorem \ref{BLMS stability}). As \cite{Huy}, we consider the group $\Aut(\D_X,\sigma)$ (resp. $\Aut_{\mathrm{s}}(\D_X,\sigma)$) of autoequivalences (resp. symplectic autoequivalences) of Fourier-Mukai type on $\D_X$ (Definition \ref{FM-type}, Definition \ref{stabilizer cubic}). The first result in our paper is as follow.

\begin{thm}\label{intromain1}
The group homomorphism \[\Aut(X) \to \Aut(\D_X), f \mapsto f_*\] induces the isomorphisms of groups 
\[\Aut(X) \iso \Aut(\D_X,\sigma), \]
\[\Aut_{\mathrm{s}}(X) \iso \Aut_{\mathrm{s}}(\D_X, \sigma),\]
where $\Aut_{\mathrm{s}}(X)$ is the symplectic automorphism group of $X$.
\end{thm}

By Theorem \ref{intromain1}, we can regard the (symplectic) automorphism groups of cubic fourfolds as symmetries of two dimensional Calabi-Yau categories. 

The second result of our paper is the comparison of automorphisms of cubic fourfolds with automorphisms of K3 surfaces. Relations between cubic fourfolds and K3 surfaces are studied via Hodge theory \cite{Has00}, derived categories \cite{Kuz} and geometry of irreducible holomorphic symplectic manifolds \cite{BD}, \cite{Add}. From point of view of Hodge theory, a labeled cubic fourfold is often related to a polarized K3 surface \cite{Has00}.  For a positive integer $d$, a labeled cubic fourfold $(X,K)$ of discriminant $d$ is a pair of a cubic fourfold $X$ and a rank two primitive sublattice $K \subset H^{2,2}(X,\ZZ)$ such that the lattice $K$ contains the square $H^2$ of the hyperplane class $H$ and the discriminant of $K$ is equal to $d$. Hassett \cite{Has00} introduced the two arithmetic conditions on an integer $d$ as follow.
 \begin{itemize}
\item[($*$):] $d>6$ and $d \equiv 0$ or $2$ (mod $6$)
\item[($**$):]  $d$ is not divisible by $4$, $9$, or any odd prime $p \equiv 2$ (mod $3$)
\end{itemize} 
Hassett \cite{Has00} proved that for a labeled cubic fourfold $(X,K)$ of discriminant $d$,  the integer $d$ satisfies $(*)$ and $(**)$ if and only if there is a polarized K3 surface $(S,h)$ of degree $d$ such that a Hodge isometry $K^\perp(-1) \simeq H^2_{\mathrm{prim}}(S,\ZZ)$ exists, where $K^\perp$ is the orthogonal complement of the sublattice $K$ in the middle cohomology group $H^4(X,\ZZ)$ of $X$ and the primitive cohomology  $H^2_{\mathrm{prim}}(S,\ZZ)$ of $S$ is the orthogonal complement of the ample class $h$ in the second cohomology $H^2(S,\ZZ)$ of $S$.
On the other hand, Kuznetsov \cite{Kuz} proposed the following conjecture for derived categories of cubic fourfolds and K3 surfaces.

\begin{conj}[\cite{Kuz}]\label{rationality}
Let $X$ be a cubic fourfold. Then $X$ is rational if and only if there is a K3 surface $S$ such that $\D_X \simeq D^b(S)$.
\end{conj} 
The rationality problem of cubic fourfolds is one of long standing problems in algebraic geometry. Conjecturally, very general cubic fourfolds are irrational.  However, there are no known irrational cubic fourfolds so far. Conjecture \ref{rationality} implies that very general cubic fourfolds are irrational. 
Addington and Thomas \cite{AT} proved that Hodge theoretical relations \cite{Has00} and categorical relations \cite{Kuz} are equivalent in some sense (See also \cite{BLMNPS}). 
In the second result of our paper, we will treat Hodge theoretical relations \cite{Has00} and categorical relations \cite{Kuz} simultaneously. For a labeled cubic fourfold $(X,K)$ and an associated K3 surface $(S,h)$ in the sense of \cite{Has00}, we will construct the isomorphism between the labeled automorphism group $\Aut(X,K)$ of the cubic fourfold $X$ and the polarized automorphism group $\Aut(S,h)$ of the K3 surface $S$ via derived categories of $X$ and $S$.
The labeled automorphism group $\Aut(X,K)$ of  a labeled cubic fourfold $(X,K)$ consists of automorphisms of the cubic fourfold $X$ acting on the sublattice $K$ identically. 
Similarly, the polarized automorphism group $\Aut(S,h)$ of a polarized K3 surface $(S,h)$ consists of automorphisms of the K3 surface $S$ fixing on the ample class $h$.
We introduce certain lattices associated to labeled cubic fourfolds and polarized K3 surfaces following \cite{AT}.
For a labeled cubic fourfold $(X,K)$, there is the rank three primitive sublattice $L_K$ of the Mukai lattice $H^*(\D_X,\ZZ)$ of the Kuznetsov component $\D_X$ such that the orthogonal complement $L^\perp_K$ in $H^*(\D_X,\ZZ)$ is Hodge isometric to $K^\perp(-1)$ (Remark \ref{label Mukai}). For a polarized K3 surface $(S,h)$, there is the rank three primitive sublattice $L_h$ of the Mukai lattice $H^*(S,\ZZ)$ of $S$ such that the orthogonal complement $L^\perp_h$ in $H^*(S,\ZZ)$ is Hodge isometric to the primitive cohomology $H^2_{\mathrm{prim}}(S,\ZZ)$ of $S$ (Section 7). The following is the second result of our paper. 
\begin{thm}\label{intromain2} 
Let $d$ be an integer satisfying $(*)$ and $(**)$. For a labeled cubic fourfold $(X,K)$ of discriminant $d$, there is a polarized K3 surface $(S,h)$ of degree $d$ such that the following hold.
\begin{itemize}
\item[(1)]There exists the object $\E \in D^b(S \times X)$ such that the Fourier-Mukai functor $\Phi_\E: D^b(S) \to \D_X$ associated to the Fourier-Mukai kernel $\E$ is an equivalence.
The cohomological Fourier-Mukai transform $\Phi^H_\E: H^*(S,\ZZ) \iso H^*(\D_X,\ZZ)$ induces the isometry 
\[\Phi^H_\E|_{L_h}: L_h \iso L_K \]
and the Hodge isometry
\[\Phi^H_\E|_{L^\perp_h}: L^\perp_h \iso L^\perp_K.\]
\item[(2)]There is a stability condition $\sigma_X \in \Stab^*(S)$ such that the group homomorphism 
\[(-)_\E: \Aut(X) \to \Aut(D^b(S),\sigma_X), f \mapsto f_\E:=\Phi^{-1}_\E \circ f_* \circ \Phi_\E \] 
is an isomorphism of groups.
Moreover, the restriction of $(-)_\E$  induces the isomorphisms 
\[(-)_\E : \Aut(X,K) \iso \Aut(S, h),\]
\[(-)_\E : \Aut_{\mathrm{s}}(X,K) \iso \Aut_{\mathrm{s}}(S, h)\]
of groups.
In particular, for any automorphism $f \in \Aut(X,K)$, there is an unique isomorphism $f_\E \in \Aut(S,h)$ such that the following diagram commutes.
\[\xymatrix{D^b(S) \ar[d]_{f_\E} \ar[r]^{\Phi_\E} & \D_X \ar[d]^{f_*} \\
D^b(S) \ar[r]^{\Phi_\E} & \D_X }  \]
\end{itemize}
\end{thm}
The key of the proof of Theorem \ref{intromain2} is the construction of a polarized K3 surface $(S,h)$ in terms of moduli spaces of stable objects in the Kuznetsov component $\D_X$ as \cite{BLMNPS}. The stability condition $\sigma_X$ in Theorem \ref{intromain2} is induced by the stability condition on $\D_X$ constructed in \cite{BLMS}. The first group isomorphism in  Theorem \ref{intromain2} (2) is deduced from Theorem \ref{intromain1}. Labeled automorphisms of a labeled cubic fourfold $(X,K)$ induce autoequivalences of the Kuznetsov component $\D_X$. Such autoequivalences induce automorphisms of the moduli space $S$ of stable objects in $\D_X$ and fix the polarization. 

The third result in our paper is about the relation between the existence of non-trivial symplectic automorphisms of cubic fourfolds and the existence of associated K3 surfaces (cf. \cite{Laza}, \cite{LZ}).

\begin{thm}[cf. Proposition 2.5, Cororally 2.9 in \cite{Laza}]\label{intromain3}
Let $X$ be a cubic fourfold. If  the symplectic automorphism group $\Aut_{\mathrm{s}}(X)$ of $X$ is not isomorphic to the trivial group $1$ or the cyclic group $\ZZ/2\ZZ$ of order $2$, there is a K3 surface $S$ such that $\D_X \simeq D^b(S)$. 
\end{thm}
The proof relies on the classification of symplectic automorphism groups of cubic fourfolds in \cite{LZ} and the lattice theoretic technique in  \cite{Nik} and \cite{Mor}. Using Theorem \ref{intromain2}, we can find examples of finite symplectic autoequivalences of K3 surfaces, which are not conjugate to symplectic automorphisms of K3 surfaces. 
For example, there is the K3 surface and the symplectic autoequivalence of order $9$  from the symplectic automorphism of order $9$ on the Fermat cubic fourfold (Example \ref{Fermat}).  
From the Klein cubic fourfold, we can construct the K3 surface and the symplectic autoequivalence of order $11$ (Example \ref{Klein}).
Typical examples of autoequivalences of derived categories of K3 surfaces are shifts, tensoring line bundles, automorphisms and spherical twists. Since shifts, tensoring line bundles and spherical twists have infinite order, it is not easy to give non-trivial examples of finite symplectic autoequivalences in terms of K3 surfaces. We can construct other examples using Theorem 1.2 and Theorem 1.8 in \cite{LZ}.

\subsection{Notation} 

We work over the complex number field $\mathbb{C}$. For a smooth projective variety $X$, we denote the bounded derived category of coherent sheaves on $X$ by $D^b(X)$. For a smooth projective variety $Z$ and an object $E \in D^b(Z)$, we define the Mukai vector $v_Z(E)$ of $E$ by 
\[ v_Z(E):=\mathrm{ch}(E) \cdot \sqrt{\mathrm{td}(Z)}.\] 
For objects $E, F \in D^b(X)$, we denote $E \otimes F:=E \otimes^\mathbf{L}F$ for simplicity.  For smooth projective varieties $X$ and $Y$ and an object $\E \in D^b(X \times Y)$, the Fourier-Mukai functor $\Phi_\E: D^b(X) \to D^b(Y)$ is defined by 
\[\Phi_\E(E):=\mathbf{R}p_*(q^*E \otimes \E),\]
where  $p:X \times Y \to Y$ and $q:X \times Y \to X$ are projections and $E \in D^b(X)$. Then the cohomological Fourier-Mukai transform $\Phi^H_\E: H^*(X,\QQ) \to H^*(Y,\QQ)$ is the linear map defined by 
\[ \Phi^H_\E(\alpha):=p_*(q^*\alpha \cdot v_{X \times Y}(\E)).\]
If $X$ and $Y$ are K3 surfaces, the the cohomological Fourier-Mukai transform  induces $\Phi^H_\E: H^*(X,\ZZ) \iso H^*(Y,\ZZ)$.
For a triangulated category $\D$ and an exceptional object $E \in \D$, the left mutation functor $\mathbf{L}_E: \D \to E^\perp$ and the $\mathbf{R}_E: \D \to {}^\perp E$ with respect to $E$ fit into the exact triangles 
\[\mathbf{R}\mathrm{Hom}(E, F) \otimes E \to F \to \mathbf{L}_E(F), \mathbf{R}_E(F) \to F \to \mathbf{R}\mathrm{Hom}(F,E)^*\otimes E \]
for $F \in \D$. 
We assume that cubic fourfolds are smooth.
We assume that K3 surfaces are projective. 
For a cubic fourfold $X$, denote the hyperplane class of $X$ by $H$. The middle primitive cohomology $H^4_{\mathrm{prim}}(X,\ZZ)$ of $X$ is the orthogonal complement $\langle H^2 \rangle^\perp$ of $H^2$ in $H^4(X,\ZZ)$.
The hyperbolic plane lattice $U$ is a lattice determined by the Gramian matrix 
\[\begin{pmatrix}
0 & 1 \\
1 & 0 \\
\end{pmatrix}.\]
The root lattice $E_8$ is positive definite.

\subsection*{Acknowledgements}
The author would like to thanks Lie Fu, Wahei Hara, Seung-Jo Jung and Naoki Koseki for valuable conversation.
The part of this work was done while the author was staying at Max Planck Institute for Mathematics in Bonn. The author is grateful to Max Planck Institute for Mathematics for its hospitality and financial support. This work is also supported by Interdisciplinary Theoretical and Mathematical Sciences Program (iTHEMS) in RIKEN and JSPS KAKENHI Grant number 19K14520. 

\section{Cubic fourfolds and K3 surfaces}
In this section, we review relations between cubic fourfolds and K3 surfaces via derived categories and Hodge theory.

\subsection{Mukai lattice of K3 surfaces}
In this subsection, we recall Mukai lattices of K3 surfaces. Mukai lattices are important in the study of derived categories of K3 surfaces.

Let $S$ be a K3 surface.   The cohomology group $H^*(S,\ZZ)$ of $S$ has a structure of a lattice given by
 \[((r_1, c_1, m_1), (r_2, c_2, m_2)):=c_1c_2-r_1m_2-r_2m_1.\] 
 Here, note that $H^*(S,\ZZ)=H^0(S,\ZZ) \oplus H^2(S,\ZZ) \oplus H^4(S,\ZZ)$.
 The lattice $(H^*(S,\ZZ),(-,-))$ is called the Mukai lattice of $S$ and $(-,-)$ is called the Mukai pairing on $S$.
 The Mukai lattice $(H^*(S,\ZZ),(-,-))$ is isometric to the even unimodular lattice $U^{\oplus4}\oplus E_8(-1)^{\oplus2}$ of signature $(4,20)$. 
 Moreover, the Mukai lattice $H^*(S,\ZZ)$ of $S$ has a weight two Hodge structure $\widetilde{H}(S)$ given by
 
\[\widetilde{H}^{2,0}(S):=H^{2,0}(S),\]
\[\widetilde{H}^{1,1}(S):=\bigoplus_{p=0}^{2}H^{p,p}(S),\]
\[\widetilde{H}^{0,2}(S):=H^{0,2}(S).\]
The integral part $\widetilde{H}^{1,1}(S,\ZZ):=\widetilde{H}^{1,1}(S) \cap H^*(S,\ZZ)$ is called the algebraic Mukai lattice of $S$. 
Note that $\widetilde{H}^{1,1}(S,\ZZ)=H^0(S,\ZZ) \oplus \mathrm{NS}(S) \oplus H^4(S,\ZZ)$. So we have $\mathrm{rk}\widetilde{H}^{1,1}(S,\ZZ)=\rho(S)+2$ and $T_S=\widetilde{H}^{1,1}(S,\ZZ)^\perp$, where $T_S$ is the transcendental lattice of $S$. For an object $E \in D^b(S)$, the Mukai vector $v(E)$ of $E$ is defined by $v(E):=\mathrm{ch}(E)\sqrt{\mathrm{td}(S)} \in \widetilde{H}^{1,1}(S,\ZZ)$. Then the homomorphism $v:K_0(S) \to \widetilde{H}^{1,1}(S,\ZZ)$ is surjective.

\subsection{Mukai lattices for cubic fourfolds}
 In this subsection, we introduce Mukai lattices for cubic fourfolds following Addington and Thomas \cite{AT}.
 
Let $X$ be a cubic fourfold. The derived category $D^b(X)$ of $X$ has a semi-orthogonal decomposition
\[D^b(X)=\langle \D_X, \OO_X, \OO_X(1), \OO_X(2) \rangle.\]
The admissible subcategory $\D_X$ is called the Kuznetsov component of $X$. Kuznetsov \cite{Kuz} proved that $\D_X$ is a $2$-dimensional Calabi-Yau category, that is the Serre functor of $\D_X$ is isomorphic to the double shift functor $[2]$.
We denote the topological K-group of $X$ by $K_{\mathrm{top}}(X)$. For an element $\alpha \in K_{\mathrm{top}}(X)$, we define the Mukai vector  $v_X(\alpha)$ of $\alpha$ as $v_X(\alpha):=\mathrm{ch}(\alpha)\sqrt{\mathrm{td}(X)} \in H^*(X,\QQ)$.
For elements $\alpha, \beta \in K_{\mathrm{top}}(X)$, we have the topological Euler characteristic $\chi_{\mathrm{top}}(\alpha, \beta) \in \ZZ$. Addington and Thomas introduced the Mukai lattice of $\D_X$.

\begin{dfn}[\cite{AT}, Definition 2.2]\label{Mukai lattice}
We define the cohomology group of $\D_X$ as 
\[H^*(\D_X, \ZZ):=\{\alpha \in K_{\mathrm{top}}(X) \mid \chi_{\mathrm{top}}([\OO_X(k)], \alpha)=0,  \text{for} \ k=0,1, 2\}. \]
Let $(-,-)$ be the restriction of $-\chi_{\mathrm{top}}(-,-)$ to $H^*(\D_X,\ZZ)$.
The pair $(H^*(\D_X,\ZZ), (-,-))$ is a lattice isometric to the even unimodular lattice $U^{\oplus4}\oplus E_8(-1)^{\oplus2}$ of signature $(4,20)$.  The lattice $(H^*(\D_X,\ZZ), (-,-))$ is called the Mukai lattice of $\D_X$. Moreover, $H^*(\D_X,\ZZ)$ has a weight two Hodge structure $\widetilde{H}(\D_X)$ given by
\[\widetilde{H}^{2,0}(\D_X):=v^{-1}_X\bigl( H^{3,1}(X)\bigr),\]
\[\widetilde{H}^{1,1}(\D_X):=v^{-1}_X\Bigl(\bigoplus_{p=0}^{4}H^{p,p}(X)\Bigr),\]
\[\widetilde{H}^{0,2}(\D_X):=v^{-1}_X\bigl( H^{1,3}(X)\bigr).\]
The integral part $\widetilde{H}^{1,1}(\D_X,\ZZ):=\widetilde{H}^{1,1}(\D_X) \cap H^*(\D_X,\ZZ)$ is called the algebraic Mukai lattice of $\D_X$. The transcendental lattice $T_{\D_X}$ of $\D_X$ is defined by the orthogonal complement $T_{\D_X}:=\widetilde{H}^{1,1}(\D_X,\ZZ)^\perp$ in $H^*(\D_X,\ZZ)$.
\end{dfn}
We consider Mukai vectors for objects in the Kuznetsov component.

\begin{rem}\label{v}

We denote the natural map $K_0(\D_X) \to H^*(\D_X,\ZZ)$ by $v$. By Proposition 2.4 in \cite{AT}, the image of $v$ is equal to the algebraic Mukai lattice $\widetilde{H}^{1,1}(\D_X,\ZZ)$. For an object $E \in \D_X$, we define the Mukai vector $v(E)$ of $E$ as 
$v(E):=v([E])$. The map $v: K_0(\D_X) \to \widetilde{H}^{1,1}(\D_X,\ZZ)$ is different from the map $v_X: K_{\mathrm{top}}(X) \to H^*(X,\QQ)$.
\end{rem}

The algebraic Mukai lattice $\widetilde{H}^{1,1}(\D_X,\ZZ)$  of $\D_X$ always contains the certain primitive sublattice of rank two. We see the definition of this sublattice.
Let $i: \D_X \to D^b(X)$ be the inclusion functor and $i^*:D^b(X) \to \D_X$ the left adjoint functor of $i$. For an integer $k \in \ZZ$, we define the element by \[\lambda_k:=[i^*\OO_{\mathrm{line}}(k)] \in \widetilde{H}^{1,1}(\D_X,\ZZ).\]
The Gramian matrix of the sublattice $\langle\lambda_1, \lambda_2 \rangle$ of $\widetilde{H}^{1,1}(\D_X,\ZZ)$ is
\[ \left( \begin{array}{cc}
2 & -1 \\
-1 & 2 
\end{array} \right),\]
Denote the sublattice $\langle\lambda_1, \lambda_2 \rangle$ of  $\widetilde{H}^{1,1}(\D_X,\ZZ)$ by $A_2$. 
Since the algebraic Mukai lattice $\widetilde{H}^{1,1}(\D_X,\ZZ)$  of $\D_X$ contains $A_2$, we have $\mathrm{rk}\widetilde{H}^{1,1}(\D_X,\ZZ)\geq 2$.
The sublattice $A_2$ is related to the primitive cohomology lattice $H^4_{\mathrm{prim}}(X,\ZZ)$ of $X$. 

\begin{prop}[\cite{AT}, Proposition 2.3]\label{prim}
Let $A_2^{\perp}$ be the orthogonal complement of $A_2$ in $H^*(\D_X,\ZZ)$. Then we have the Hodge isometry $v_X: A_2^{\perp} \iso H^4_{\mathrm{prim}}(X,\ZZ)(-1)$.
\end{prop}

\subsection{Picard numbers of Kuznetsov components and K3 surfaces}
In this subsection, we introduce Picard numbers of Kuznetsov components.

 Let $X$ be a cubic fourfold.
The following is the definition of the Picard number of the Kuznetsov component $\D_X$.

\begin{dfn}
The Picard number $\rho(\D_X)$ of the Kuznetsov component $\D_X$ of $X$ is
\[ \rho(\D_X):=\mathrm{rk}\widetilde{H}^{1,1}(\D_X,\ZZ)-2. \]
\end{dfn}
Since $\mathrm{rk}\widetilde{H}^{1,1}(\D_X,\ZZ)\geq 2$, we have $\rho(\D_X)\geq 0$.
We can describe the Picard number of $\D_X$ in terms of Hodge structure on $H^4(X,\ZZ)$.

\begin{rem}\label{2,2}
The equality $\rho(\D_X)=\mathrm{rk}H^{2,2}(X,\ZZ)-1$ holds.
\end{rem}
\begin{proof}
It is deduced from Definition \ref{Mukai lattice} and Proposition \ref{prim}.
\end{proof}

Picard numbers of Kuznetsov components are compatible with Picard numbers of K3 surfaces.

\begin{rem}\label{Picard number}
If there is a K3 surface $S$ such that $\D_X \simeq D^b(S)$, we have a Hodge isometry 
${H}^*(\D_X,\ZZ) \simeq {H}^*(S,\ZZ)$. So we obtain $\rho(\D_X)=\rho(S)\geq 1$ and $T_{\D_X} \simeq T_S$. 
\end{rem}

For a cubic fourfold $X$, there is a necessary and sufficient condition of the existence of a K3 surface $S$ such that $\D_X \simeq D^b(S)$ in terms of lattice theory.

\begin{thm}[\cite{AT}, \cite{BLMNPS}]\label{derivedTorelli}
There is an embedding of the hyperbolic lattice $U$ into the algebraic Mukai lattice $\widetilde{H}^{1,1}(\D_X,\ZZ)$ of $\D_X$ if and only if there is a K3 surface such that $\D_X \simeq D^b(S)$. 
\end{thm}
In Theorem \ref{derivedTorelli}, the hyperbolic lattice $U$ is corresponding to the sublattice $H^0(S,\ZZ)\oplus H^4(S,\ZZ)$ of the algebraic Mukai lattice $\widetilde{H}^{1,1}(S,\ZZ)$ of a K3 surface $S$.

\subsection{Special cubic fourfolds}
In this subsection, we recall the notion of special cubic fourfolds following Hassett \cite{Has00}.

Let $X$ be a cubic fourfold. If there is a K3 surface $S$ such that $\D_X \simeq D^b(S)$, then we have $\rho(\D_X)\geq1$ by Remark \ref{Picard number}. A cubic fourfold $X$ is special if we have $\rho(\D_X)\geq1$, equivalently $\mathrm{rk}H^{2,2}(X,\ZZ)\geq2$ by Remark \ref{2,2}. The following is the analogue of the notion of polarized K3 surfaces.  

\begin{dfn}[\cite{Has00}]\label{labeled cubic fourfold}
For a positive integer $d$, a labeled cubic fourfold $(X,K)$ of discriminant $d$ is a pair of a special cubic fourfold $X$ and a rank two primitive sublattice $K \subset H^{2,2}(X,\ZZ)$ such that $K$ contains $H^2$ and $\mathrm{disc}(K)=d$. A cubic fourfold $X$ is a special cubic fourfold of discriminant $d$ if there is a rank two primitive sublattice $K \subset H^{2,2}(X,\ZZ)$ such that $(X,K)$ is a labeled cubic fourfold of discriminant $d$.
\end{dfn}
We interpret labeled cubic fourfolds of discriminant $d$ via Mukai lattices.

\begin{rem}[cf. Subsection 2.4 in \cite{AT}]\label{label Mukai}
Let $(X,K)$ be a labeled cubic fourfold of discriminant $d$. We define the rank three primitive sublattice $L_K$ of the algebraic Mukai lattice $\widetilde{H}^{1,1}(\D_X,\ZZ)$ such that $\mathrm{disc}L_K=d$ as follow. 
There is a class $T \in H^{2,2}_{\mathrm{prim}}(X,\ZZ)$ such that $K \cap H^{2,2}_{\mathrm{prim}}(X,\ZZ)=\ZZ \cdot T$. Using Proposition \ref{prim}, we define the class $\kappa_T \in A^\perp_2$  by  $\kappa_T:=v^{-1}_X(T)$. Let $L_K \subset \widetilde{H}^{1,1}(\D_X,\ZZ)$ be the saturation of the sublattice generated by $A_2$ and $\kappa_T$. By Proposition \ref{prim} and the definition of $L_K$, we have the Hodge isometry $v_X: L_K^\perp \iso K^\perp(-1)$, where we take orthogonal complements in $H^*(\D_X,\ZZ)$ and $H^4(X,\ZZ)(-1)$ respectively. Since $K$ and $L_K$ are primitive sublattices of unimodular lattices $H^4(X,\ZZ)$ and $H^*(\D_X,\ZZ)$ respectively, we have 
\[d=\mathrm{disc}(K)=\mathrm{disc}\bigr(K(-1)\bigl)=\mathrm{disc}\bigl(K^\perp(-1)\bigr)=\mathrm{disc}(L_K). \]
\end{rem}

The moduli space $\C$ of cubic fourfolds is a twenty dimensional quasi-projective variety. For a positive integer $d$, denote the subset of special cubic fourfolds of discriminant $d$ by $\C_d$. Hassett \cite{Has00} introduced the two arithmetic conditions on an integer $d$ as follow.
 \begin{itemize}
\item[($*$):] $d>6$ and $d \equiv 0$ or $2$ (mod $6$)
\item[($**$):]  $d$ is not divisible by $4$, $9$, or any odd prime $p \equiv 2$ (mod $3$)
\end{itemize} 
Hassett \cite{Has00} proved that the condition ($*$) is equivalent to the non-emptyness of $\C_d$. 

\begin{rem}[\cite{Has00}]
If an integer $d$ satisfies $(*)$, then the subset $\C_d$ is a subvariety of codimension one in $\C$.
\end{rem}

The condition $(**)$ is related to K3 surfaces.
For labeled cubic fourfolds and polarized K3 surfaces, Hassett \cite{Has00} proved the following theorem.

\begin{thm}[\cite{Has00}]\label{Hassett K3}
Assume that an integer satisfies $(*)$. Let $(X,K)$ be a labeled cubic fourfold of discriminant $d$. The integer $d$ satisfies $(**)$ if and only if there is a polarized K3 surface $(S,h)$ of degree $d$ such that we have a Hodge isometry $K^\perp(-1) \simeq h^\perp$. Here, we take orthogonal complements of $K$ and $h$ in $H^4(X,\ZZ)$ and $H^2(S,\ZZ)$ respectively.
\end{thm}

In the context of derived categories, the following is known.
\begin{thm}[\cite{AT}, \cite{BLMNPS}]
Let $X$ be a cubic fourfold. There is a K3 surface $S$ such that $\D_X \simeq D^b(S)$ if and only if there is an integer $d$ satisfying $(*)$ and $(**)$ such that $X \in \C_d$.
\end{thm}

\section{Stability conditions}
In this section, we review facts about stability conditions on Kuznetsov components of cubic fourfolds and derived categories of K3 surfaces.

\subsection{Weak stability conditions}
In this subsection, we introduce the notion of (weak) stability conditions following \cite{Bri07} and \cite{BLMS}.

Let $\D$ be a triangulated category over $\CC$. The definition of weak stability conditions and stability conditions on $\D$ is the following.

\begin{dfn}
Fix a finitely generated free abelian group $\Lambda$ and a surjective group homomorphism $\mathrm{cl}:K_0(\D) \to \Lambda$. A weak stability condition on $\D$  \rm{(}\it{}with respect to $\Lambda$\rm{)} is a pair $\sigma=(Z,\A)$ of a group homomorphism $Z: \Lambda \to \CC$  \rm{(}\it{}called a central charge\rm{)} \it{}and the heart of a bounded t-structure $\A$ in $\D$ such that the following three properties hold.
\begin{itemize}
\item[(i)] For any object $E \in \A$, we have $\mathrm{Im}Z(\mathrm{cl}(E))\geq 0$ and if $\mathrm{Im}Z(\mathrm{cl}(E))=0$, then $Z(\mathrm{cl}(E)) \in \RR_{\leq0}$ holds. For simplicity, we will denote $Z(\mathrm{cl}(E))$ by $Z(E)$ for an object $E \in \D$.

We prepare terminologies to state \rm{}(ii) \it{}and \rm{}(iii)\it{}. For $E \in \A$ with $\mathrm{Im}Z(E)>0$, we define the slope $\mu_\sigma(E)$ of $E$ with respect to $\sigma$ as 
\[ \mu_\sigma(E):=-\frac{\mathrm{Re}Z(E)}{\mathrm{Im}Z(E)}. \]
For $E \in \A$ with $\mathrm{Im}Z(E)=0$, we put $\mu_\sigma(E):=\infty$. For a nonzero object $E \in \A$, $E$ is $\sigma$-semistable if for any subobject $F$ of $E$ in $\A$, we have $\mu_\sigma(F)\leq \mu_\sigma(E)$.
\item[(ii)]For any $E \in \A$, there exists a filtration
\[0=E_0 \subset E_1 \subset \cdot \cdot \cdot \subset E_n=E\]
in $\A$ such that for any $1 \leq k \leq n$, the quotient $F_k:=E_k/E_{k-1}$ is $\sigma$-semistable with 
\[ \mu_\sigma(F_1)>\mu_\sigma(F_2)>\cdot \cdot \cdot >\mu_\sigma(F_n) .\]
This filtration is called a Harder-Narasimhan filtration of $E$ with respect to $\sigma$. 
 
\item[(iii)]There exists a quadratic form $Q$ on $\Lambda \otimes \RR$ such that $Q|_{\mathrm{Ker}Z}$ is negative definite and $Q(\mathrm{cl}(E))\geq0$ for any $\sigma$-semistable object $E \in \A$. This property is called the support property.
\end{itemize}
A weak stability condition $\sigma=(Z,\A)$ on $\D$ is a stability condition on $\D$ if $Z(E)\neq 0$ holds for any nonzero object $E \in \A$.
\end{dfn}

\begin{rem}[\cite{BLMS}, Remark 2.6]
Let $\sigma=(Z,\A)$ be a weak stability condition on $\D$ with respect to $\Lambda$. If $\mathrm{rk}\Lambda=2$ and $Z:\Lambda \to \CC$ is injective, any positive semi-definite quadratic form $Q$ on $\Lambda \otimes \RR$ satisfies the support property. 
\end{rem}

Bridgeland \cite{Bri07} proved that the set $\Stab(\D)$ of stability conditions on $\D$ with respect to $\Lambda$ has a structure of a complex manifold.

\begin{rem}[\cite{Bri07}]
Let $\sigma=(Z,\A)$ be a stability condition on $\D$ with respect to $\Lambda$. For an object $E \in \A \setminus\{0\}$, we define the phase $\phi_\sigma(E)$ of $E$ with respect to $\sigma$ by
\[\phi_{\sigma}(E):=\frac{1}{\pi}\mathrm{arg}Z(E) \in (0, 1].\]
For a real number $\phi \in (0,1]$, we define the full subcategory $\P_\sigma(\phi)$ in $\A$ by
\[\P_\sigma(\phi):=\{E \in \A \mid \text{$E$ is $\sigma$-semistable with $\phi_\sigma(E)=\phi$}\} \cup \{0\}.\] 
For any real number $\phi \in (0,1]$, the full subcategory $\P_\sigma(\phi)$ is an abelian subcategory of $\A$.
\end{rem}

We use weak stability conditions to construct new hearts of bounded t-structures. By the existence of Harder-Narasimhan filtrations, we have the following torsion pairs and they produce new hearts of bounded t-structures. 
\begin{dfn}\label{tilt}
Let $\sigma=(Z,\A)$ be a weak stability condition on $\D$. For $\mu \in \RR$, we define a torsion pair $(\T^\mu_\sigma, \F^\mu_\sigma)$ on $\A$ as 
\[\T^\mu_\sigma:=\langle E \in \A \mid \text{$E$ is $\sigma$-semistable with $\mu_\sigma(E)>\mu \rangle$} ,\]
\[ \F^\mu_\sigma:=\langle E \in \A \mid \text{$E$ is $\sigma$-semistable with $\mu_\sigma(E) \leq \mu$} \rangle,\]
where $\langle - \rangle$ is the extension closure. We define the heart $\A^{\mu}_\sigma$ of a bounded t-structure as 
\[ \A^\mu_\sigma:=\langle \F^\mu_\sigma[1], \T^\mu_\sigma \rangle.\]
We say that $\A^\mu_\sigma$ is obtained by the tilting of $\A$ with respect to the torsion pair $(\T^\mu_\sigma, \F^\mu_\sigma)$.
\end{dfn}

\subsection{Stability conditions on K3 surfaces}
In this subsection, we recall examples of stability conditions on derived categories of K3 surfaces. 

Let $S$ be a K3 surface. Using the group homomorphism $v: K_0(S) \to \widetilde{H}^{1,1}(S,\ZZ)$, we consider only stability conditions on $D^b(S)$ with respect to $\widetilde{H}^{1,1}(S,\ZZ)$. Take $\RR$-divisors $\beta, \omega \in \mathrm{NS}(S)_{\RR}$ such that $\omega$ is an ample class. The weak stability condition $\sigma_\omega=(Z_\omega, \mathrm{Coh}(S))$ is given by 
\[Z_\omega(E):=i \mathrm{ch}_0(E)-\mathrm{ch}_1(E)\cdot \omega\]
for an object $E \in D^b(S)$. It is nothing but the slope stability on $\mathrm{Coh}(S)$. 
Using Definition \ref{tilt}, we define the heart $\A_{\beta,\omega}$ of a bounded t-structure on $D^b(S)$ by
\[\A_{\beta,\omega}:=\mathrm{Coh}(S)^{\beta \cdot \omega}_{\sigma_{\omega}}.\]
Let $Z_{\beta,\omega}: \widetilde{H}^{1,1}(S,\ZZ) \to \CC$ be the group homomorphism defined by
\[Z_{\beta,\omega}(w):=(e^{\beta+i\omega},w),\] 
where $w \in \widetilde{H}^{1,1}(S,\ZZ)$. 
Bridgeland \cite{Bri08} proved the following theorem.

\begin{thm}[Lemma 6.2 in \cite{Bri08}]\label{stability K3}
If $Z_{\beta, \omega}(E) \notin \RR_{<0}$ holds for any spherical sheaf $E$ on $S$, then the pair $\sigma_{\beta,\omega}:=(Z_{\beta,\omega}, \A_{\beta,\omega})$ is a stability condition on $D^b(S)$. If $\omega^2>2$,  we have $Z_{\beta, \omega}(E) \notin \RR_{<0}$ for any spherical sheaf $E$ on $S$.\end{thm}

\subsection{Spaces of stability conditions on K3 surfaces}
In this subsection, we recall structures of spaces of stability conditions on K3 surfaces.

Let $S$ be a K3 surface. Let $\Stab(S)$ be the space of stability conditions on $D^b(S)$ with respect to $\widetilde{H}^{1,1}(S, \ZZ)$. 
We define the subset $\P(S)$ of $\widetilde{H}^{1,1}(S, \ZZ)\otimes \CC$  as 
 \[\P(S):=\{\Omega \in \widetilde{H}^{1,1}(S, \ZZ)\otimes \CC \mid \text{$\langle \mathrm{Re}(\Omega), \mathrm{Im}(\Omega) \rangle_{\mathbb{R}}$ is a positive definite plane} \}.\] 
Let $\P^+(S)$ be the connected component of $\P(S)$ containing $e^{i\omega}$, where $\omega$ is an ample divisor on $S$. 
Let $\Delta_S$ be the set of $(-2)$-classes in $\widetilde{H}^{1,1}(S, \ZZ)$. We define the subset $\P^+_0(S)$ of $\P^+(S)$ as 
\[\P^+_0(S):=\P^+(S)\setminus \bigcup_{\delta \in \Delta_S} \delta^\perp.\]
We consider the action of the autoequivalence group $\Aut(D^b(S))$ on $\Stab(S)$.
\begin{dfn}
For an autoequivalence $\Phi \in \Aut(D^b(S))$ and a stability condition\\ $\sigma=(Z,\A) \in \Stab(S)$, we define the stability condition $\Phi\sigma$ by
 \[\Phi\sigma:=(Z \circ (\Phi^H)^{-1}, \Phi(\A)) \in \Stab(S). \]
\end{dfn}
 
Since the Mukai pairing $(-,-)$ on $\widetilde{H}^{1,1}(S, \ZZ)$ is non-degenerate, for a stability condition $\sigma=(Z,\A) \in \mathrm{Stab}(S)$, there is the unique element $\Omega_Z \in \widetilde{H}^{1,1}(S, \ZZ)\otimes \CC$ such that $Z(-)=(\Omega_Z,-)$.  Let $\Stab^*(S) \subset \Stab(S)$ be the connected component containing stability conditions in Theorem \ref{stability K3}.
Bridgeland \cite{Bri08} proved the following theorem.

\begin{thm}[\cite{Bri08}]\label{BriK3}
For a stability condition $\sigma=(Z,\A) \in \Stab^*(S)$, put $\pi_S(\sigma):=\Omega_Z$.
Then $\pi_S$ induces the covering map $\pi_S: \Stab^*(S) \to \P^+_0(S)$. We define 
\[\Aut^0(D^b(S)):=\{\Phi \in \Aut(D^b(S)) \mid \Phi^H=\mathrm{id}, \Phi(\Stab^*(S)) \subset \Stab^*(S)\}.\]
Then the natural homomorphism $\Aut^0(D^b(S)) \to \mathrm{Deck}(\pi_S)$ is an isomorphism, where $\mathrm{Deck}(\pi_S)$ is the group of deck transformations of the covering map $\pi_S$.
\end{thm}

\subsection{Clifford algebra associated to a line on a cubic fourfold}
In this subsection, we recall the Clifford algebra on the projective space $\mathbb{P}^3$ associated to a line on a cubic fourfold in \cite{Kuz08} and Section 7 in \cite{BLMS}. 

Let $X$ be a cubic fourfold. Take a line $l \subset X$. Consider the blowing up $p_l: \mathrm{Bl}_lX \to X$ of $X$ along the line $l$. We have the embedding $j_l: \mathrm{Bl}_lX \to \mathrm{Bl}_l\mathbb{P}^5$, where $\mathrm{Bl}_l\mathbb{P}^5 \to \mathbb{P}^5$ is the blowing up of $\mathbb{P}^5$ along the line $l$. The linear projection from $l$ induces the following commutative diagram.
\[ \xymatrix{ &\mathrm{Bl}_lX \ar@{^{(}-{>}} [r]^{j_l} \ar[d]_{p_l}& \mathrm{Bl}_l{\mathbb{P}^5} \ar[d] \ar[rd]^{q_l} && \\ l \ar@{^{(}-{>}} [r] & X \ar@{^{(}-{>}} [r] & \mathbb{P}^5 \ar@{.{>}} [r]& \mathbb{P}^3  & } \]
Denote $h:=\mathrm{c}_1(\OO_{\mathbb{P}^3}(1))$. Note that $q_l:\mathrm{Bl}_l\mathbb{P}^5 \to \mathbb{P}^3$ is a $\mathbb{P}^2$-bundle via the isomorphism $\mathrm{Bl}_l\mathbb{P}^5 \simeq  \mathbb{P}(\OO_{\mathbb{P}^3}^{\oplus2} \oplus \OO_{\mathbb{P}^3}(-h))$. The composition $\pi_l:=q_l \circ j_l: \mathrm{Bl}_lX \to \mathbb{P}^3$ is a conic fibration. Let $\mathcal{B}^l_0$ (resp. $\mathcal{B}^l_1$) be the even part (resp. the odd part) of the sheaf of Clifford algebras on $\mathbb{P}^3$ associated to $\pi_L$. For $k \in \ZZ$, we define the $\mathcal{B}^l_0$-bimodule $\B^l_k$ by $\B^l_k=\mathcal{B}^l_{k-2} \otimes \OO_{\mathbb{P}^3}(h)$. Let $\mathrm{Coh}(\mathbb{P}^3, \mathcal{B}^l_0)$ be the category of coherent right $\mathcal{B}^l_0$-modules and define $D^b(\mathbb{P}^3, \mathcal{B}^l_0):=D^b(\mathrm{Coh}(\mathbb{P}^3, \mathcal{B}^l_0))$. Note that $\mathbf{L}p_l^*: D^b(X) \to D^b(\mathrm{Bl}_lX)$ is a fully faithful functor. There exists a coherent sheaf $\mathcal{E}_l$ of right $\pi^*_l\mathcal{B}^l_0$-modules on $\mathrm{Bl}_lX$ such that $\Psi_l(-):=\mathbf{R}{\pi_l}_*(-\otimes \OO_{\mathrm{Bl}_lX}(h) \otimes \mathcal{E}_l)[1]$ is a fully faithful functor from $\mathbf{L}p_l^*\D_X$ to $D^b(\PP^3,\mathcal{B}^l_0)$ and there is a semi-orthogonal decomposition $D^b(\PP^3,\mathcal{B}^l_0)=\langle \Psi_l(\mathbf{L}p_l^*\D_X), \B^l_1, \B^l_2, \B^l_3 \rangle$. (See \cite{BLMS}, Proposition 7.7.)
\begin{rem}\label{exact}
By Section 4 in \cite{Kuz08}, the coherent right $\B^L_0$-module $\E_L$ fits into the exact sequence
\[ 0 \to q^*_l\B^l_{-1}(-2H) \to q^*_l\B^l_0(-H) \to j_{l*}\E_l \to 0.\]
\end{rem}

\subsection{Stability conditions on Kuznetsov components}
In this subsection, we recall the examples of stability conditions on Kuznetsov components of cubic fourfolds in \cite{BLMS}. 

Let $X$ be a cubic fourfold. Using the group homomorphism $v: K_0(\D_X) \to \widetilde{H}^{1,1}(\D_X, \ZZ)$ in Remark \ref{v}, we consider only stability conditions on $\D_X$ with respect to $\widetilde{H}^{1,1}(\D_X, \ZZ)$.  Fix a line $l$ on $X$. 
We define the Chern character map $\ch_{\mathcal{B}^l_0} : K_0(D^b(\mathbb{P}^3, \mathcal{B}^l_0)) \to H^*(\mathbb{P}^3, \QQ)$ as
\[ \ch_{\mathcal{B}^l_0}(E):=\ch(\mathrm{forg}(E))\Bigl(1-\frac{11}{32}h^2\Bigr), \]
where $E\in D^b(\mathbb{P}^3,\mathcal{B}^l_0)$ and $\mathrm{forg}:D^b(\mathbb{P}^3, \mathcal{B}^l_0) \to D^b(\mathbb{P}^3)$ is the forgetful functor.
For $\beta \in \RR$, we define the twisted Chern character map $\ch^{\beta}_{\mathcal{B}^l_0}: K_0(D^b(\mathbb{P}^3, \mathcal{B}^l_0)) \to H^*(\mathbb{P}^3, \RR)$ as 
$\chb:=e^{-\beta h}\ch_{\mathcal{B}^l_0}$. The chern character map $\ch_{\mathcal{B}^l_0}$ is same as $\ch^0_{\B^l_0}$. Using the isomorphism 
\[ \RR^4 \iso H^*(\PP^3,\RR), (x_1,x_2,x_3,x_4) \mapsto (x_1, x_2h, x_3h^2, x_4h^3),\] 
we regard $\chb(E)=(\ch^{\beta}_{\mathcal{B}^l_0,0}(E),\chbi(E),\chbii(E),\chbiii(E)) \in \RR^4$  for $E\in D^b(\PP^3,\mathcal{B}^l_0)$.  For $j=1,2$, we define the finitely generated free abelian group $\Lambda^j_{\mathcal{B}^l_0}$ as 
\[\Lambda^j_{\mathcal{B}^l_0}:=\langle \ch_{\mathcal{B}^l_0,k}(E)\mid E \in D^b(\PP^3,\mathcal{B}^l_0),0 \leq k \leq j \rangle_\ZZ \subset \RR^4. \]
Note that $\mathrm{rk}\Lambda^j_{\mathcal{B}^l_0}=1+j$. We define a weak stability condition $\sigma_{\mathrm{slope}}=(Z_{\mathrm{slope}},\mathrm{Coh}(\PP^3,\mathcal{B}^l_0))$ on $D^b(\PP^3,\mathcal{B}^l_0)$ with respect to $\Lambda^j_{\mathcal{B}^l_0}$ as 
\[ Z_{\mathrm{slope}}(E):=i \ch_{\mathcal{B}^l_0,0}(E)-\ch_{\mathcal{B}^l_0,1}(E) \]
for $E \in D^b(\PP^3,\B^l_0)$. It is nothing but the slope stability. As Definition \ref{tilt}, for $\beta \in \RR$, consider the heart  $\mathrm{Coh}^\beta(\PP^3,\B^l_0)$ of a bounded t-structure in $D^b(\PP^3,\B^l_0)$ defined by 
\[\mathrm{Coh}^\beta(\PP^3,\B^l_0) :=\mathrm{Coh}(\PP^3,\B^l_0)^{\beta}_{\sigma_{\mathrm{slope}}}.\] For $\alpha>0$ and $\beta \in \RR$, we can define a weak stability condition $\sigma_{\alpha, \beta}=(Z_{\alpha,\beta}, \mathrm{Coh}^\beta(\PP^3,\B^l_0))$ with respect to $\Lambda^2_{\mathcal{B}^l_0}$ by 
\[ Z_{\alpha, \beta}(E):=i\chbi(E)+\frac{1}{2}\alpha^2\chbo(E)-\chbii(E)\]
for $E \in D^b(\PP^3,\B^l_0)$. (See \cite{BLMS}, Proposition 9.3.)
Consider the heart $\mathrm{Coh}^0_{\alpha,\beta}(\PP^3,\B^l_0)$ of a bounded t-structure on $D^b(\PP^3,\B^l_0)$, that is defined by 
\[\mathrm{Coh}^0_{\alpha,\beta}(\PP^3,\B^l_0):=\mathrm{Coh}^\beta(\PP^3,\B^l_0))^0_{\sigma_{\alpha,\beta}}.\]
Let $\widetilde{\Psi_l}: \D_X \iso \Psi(\mathbf{L}p_l^*\D_X)$ be the equivalence induced by the fully faithful functor $\Psi_l \circ \mathbf{L}p_l^* :\D_X \to D^b(\PP^3, \B^l_0)$.

\begin{thm}[\cite{BLMS}, Proposition 2.6 in \cite{LPZ}]\label{BLMS stability}
For $0<\alpha<1/4$, we define 
\[ \A^l_\alpha:=\widetilde{\Psi_{l}}^{-1}(\mathrm{Coh}^0_{\alpha,-1}(\PP^3,\B^l_0) \cap \widetilde{\Psi_l}(\D_X)), \]
\[Z^l_\alpha(E):=-iZ_{\alpha,-1}(\widetilde{\Psi_l}(E)), E \in \D_X. \]
Then $\sigma^l_\alpha=(Z^l_\alpha, \A^l_\alpha)$ is a stability condition on $\D_X$ with respect to $\widetilde{H}^{1,1}(\D_X, \ZZ)$. For any lines $l$ and $l'$ on $X$,  we have $\sigma^l_\alpha=\sigma^{l'}_\alpha$. Denote the stability condition $\sigma^l_\alpha$ on $\D_X$ by $\sigma_\alpha=(Z_\alpha, \A_\alpha)$.
\end{thm}

\subsection{Spaces of stability conditions on Kuznetsov components}
In this subsection, we recall structures of spaces of stability conditions on Kuznetsov components.

Let $X$ be a cubic fourfold. Let $\Stab(\D_X)$ be the space of stability conditions on $\D_X$ with respect to $\widetilde{H}^{1,1}(\D_X, \ZZ)$. We define the subset $\P(\D_X)$ of $\widetilde{H}^{1,1}(\D_X, \ZZ)\otimes \CC$  as 
\[\P(\D_X):=\{\Omega \in \widetilde{H}^{1,1}(\D_X, \ZZ)\otimes \CC \mid \text{$\langle \mathrm{Re}(\Omega), \mathrm{Im}(\Omega) \rangle_{\mathbb{R}}$ is a positive definite plane} \}.\] 
Since the Mukai pairing $(-,-)$ on $H^*(\D_X,\ZZ)$ is non-degenerate, for a stability condition $\sigma=(Z,\A) \in \mathrm{Stab}(\D_X)$, there is the unique element $\Omega_Z \in \widetilde{H}^{1,1}(\D_X, \ZZ)\otimes \CC$ such that $Z(-)=(\Omega_Z,-)$.  Let $\P^+(\D_X)$ be the connected component of $\P(\D_X)$ containing $\Omega_{Z_\alpha}$ for $0<\alpha <1/4$. Here, $\sigma_\alpha$ is the stability condition in Theorem \ref{BLMS stability}. Let $\Delta_X$ be the set of $(-2)$-classes in $\widetilde{H}^{1,1}(\D_X, \ZZ)$. We define the subset $\P^+_0(\D_X)$ of $\P^+(\D_X)$ as 
\[\P^+_0(\D_X):=\P^+(\D_X)\setminus \bigcup_{\delta \in \Delta_X} \delta^\perp.\]
As the case of K3 surfaces, we consider the action of the autoequivalence group of $\D_X$ on $\Stab(\D_X)$. To consider the action, we need the following.

\begin{dfn}\label{FM-type}
An autoequivalence $\Phi:\D_X \to \D_X$ is called Fourier-Mukai type if there exists $\mathcal{E} \in D^b(X \times X)$ such that the following diagram commutes.
\[\xymatrix{D^b(X) \ar[d]_{i^*} \ar[r]^{\Phi_{\mathcal{E}}} & D^b(X) \\
\D_X \ar[r]^{\Phi} & \D_X \ar[u]_{i}}  \]
Here, $\Phi_{\mathcal{E}}$ is the Fourier-Mukai functor with the Fourier-Mukai kernel $\mathcal{E}$.
We denote the group of autoequivalences of Fourier-Mukai type by $\Aut^{\FM}(\D_X)$. For $\Phi \in \Aut^{\FM}(\D_X)$, $\Phi^H: H^*(\D_X, \ZZ) \to H^*(\D_X,\ZZ)$ is the Hodge isometry induced by $\Phi$.
\end{dfn}
We will consider only autoequivalences of Fourier-Mukai type.

\begin{dfn}
For an autoequivalence $\Phi \in \Aut^{\mathrm{FM}}(\D_X)$ and a stability condition $\sigma=(Z,\A) \in \Stab(\D_X)$, we define the stability condition $\Phi\sigma:=(Z \circ (\Phi^H)^{-1}, \Phi(\A)) \in \Stab(\D_X)$. 
\end{dfn}

Let $\mathrm{Stab}^*(\D_X) \subset \Stab(\D_X)$ be the connected component containing stability conditions in Theorem \ref{BLMS stability}. Then the following holds. 
\begin{thm}[\cite{BLMS}, \cite{BLMNPS}]
For a stability condition $\sigma=(Z,\A) \in \Stab^*(\D_X)$, put $\pi_X(\sigma):=\Omega_Z$.
Then $\pi_X$ induces the covering map $\pi_X: \Stab^*(\D_X) \to \P^+_0(\D_X)$.
\end{thm} 

\section{Autoequivalences of K3 surfaces and stability conditions}
In this section, we introduce certain subgroups of autoequivalence groups of K3 surfaces related to stability conditions following \cite{Huy}. We will study polarized automorphisms of K3 surfaces and their relation with stability conditions.

\subsection{Subgroups of autoequivalence groups of K3 surfaces}
In this subsection, we see the definition of groups that we are interested in. 

Let $S$ be a K3 surface. First, we recall the notion of symplectic automorphism group $\Aut_{\mathrm{s}}(S)$ of $S$ and the symplectic autoequivalence group $\Aut_{\mathrm{s}}(D^b(S))$. 

\begin{dfn}
An automorphism $f \in \Aut(S)$ of $S$ is a symplectic automorphism of $S$ if the pullback $f^*$ of $f$ acts on $H^{2,0}(S)$ trivially.
The group $\Aut_{\mathrm{s}}(S)$ of symplectic automorphisms of $S$ is called the symplectic automorphism group of $S$. An autoequivalence $\Phi \in \Aut(D^b(S))$ is a symplectic autoequivalence of $D^b(S)$ if the cohomological Fourier-Mukai transform $\Phi^H$ acts on $\widetilde{H}^{2,0}(S)$ trivially. The group $\Aut_{\mathrm{s}}(D^b(S))$ of symplectic autoequivalences of $D^b(S)$ is called the symplectic autoequivalence group of $D^b(S)$. 
\end{dfn}

 Huybrechts \cite{Huy} studied the following subgroups of $D^b(S)$.

\begin{dfn}\label{stabilizer K3}
For a stability condition $\sigma \in \Stab^*(S)$, we define the group $\Aut(D^b(S),\sigma)$ of autoequivalences fixing $\sigma$ by
\[\Aut(D^b(S),\sigma):=\{\Phi \in \Aut(D^b(S)) \mid \Phi\sigma=\sigma\}.\]
Denote the intersection of $\Aut(D^b(S),\sigma)$ and $\Aut_{\mathrm{s}}(D^b(S))$ by $\Aut_{\mathrm{s}}(D^b(S),\sigma)$.
\end{dfn}

The groups in Definition \ref{stabilizer K3} can be described in terms of Mukai latices.

\begin{dfn}\label{stabilizer Hodge K3}
For a linear subspace $W \subset H^*(S,\RR)$, we define 
\[\mathrm{O}_{\mathrm{Hodge}}(H^*(S,\ZZ),W):=\{\varphi \in \mathrm{O}_{\mathrm{Hodge}}(H^*(S,\ZZ)) \mid \varphi |_W=\mathrm{id}_W\}.\]
The weight two Hodge structure on $H^*(S,\ZZ)$ defines the positive definite plane $P_S$ given by 
\[P_S:=\bigl(\widetilde{H}^{2,0}(S)\oplus\widetilde{H}^{0,2}(S)\bigr)\cap H^*(S,\RR).  \]
For a stability condition $\sigma=(Z,\A) \in \Stab^*(S)$, we have the positive definite plane $P_\sigma$ defined by
\[P_\sigma:=\langle \mathrm{Re}\bigl(\pi_S(\sigma)\bigl), \mathrm{Im}\bigl(\pi_S(\sigma)\bigr) \rangle_{\mathbb{R}}.\] 
For a stability condition $\sigma \in \Stab^*(S)$, we put
\[ \Pi_\sigma:=P_S \oplus P_\sigma.\] 

\end{dfn}

Huybrechts proved the following proposition. 

\begin{prop}[\cite{Huy}]\label{symmetry lattice}
Let $\sigma \in \Stab^*(S)$ be a stability condition on $D^b(S)$. 
Then we have the isomorphism 
\[(-)^H:\Aut(D^b(S),\sigma) \iso \mathrm{O}_{\mathrm{Hodge}}(H^*(S,\ZZ),P_\sigma), \Phi \mapsto \Phi^H \]
of groups $\Aut(D^b(S),\sigma)$ and $\mathrm{O}_{\mathrm{Hodge}}(H^*(S,\ZZ),P_\sigma)$. The restriction of this isomorphism to $\Aut_{\mathrm{s}}(D^b(S),\sigma)$ induces the isomorphism
\[(-)^H: \Aut_{\mathrm{s}}(D^b(S),\sigma) \iso \mathrm{O}_{\mathrm{Hodge}}(H^*(S,\ZZ),\Pi_\sigma) \]
of groups $\Aut_{\mathrm{s}}(D^b(S),\sigma)$ and $\mathrm{O}_{\mathrm{Hodge}}(H^*(S,\ZZ),\Pi_\sigma)$. 
\end{prop}

\subsection{Polarized automorphisms of K3 surfaces}
In this subsection, we study stability conditions on K3 surfaces fixed by polarized automorphisms. 
First, we define polarized automorphisms of K3 surfaces.
\begin{dfn}
For a polarized K3 surface $(S,h)$, we define the group $\Aut(S,h)$ of automorphisms of $(S,h)$ by
\[\Aut(S,h):=\{f \in \Aut(S) \mid f^*h=h \}.\]
An automorphism $f \in \Aut(S)$ of a K3 surface $S$ is called a polarized automorphism of $S$ if there is a primitive ample divisor $h \in \mathrm{NS}(S)$ such that $f \in \Aut(S,h)$.
Denote the intersection of $\Aut(S,h)$ and  $\Aut_{\mathrm{s}}(S)$ by $\Aut_{\mathrm{s}}(S,h)$.
\end{dfn} 

Let $(S,h)$ be a polarized K3 surface. Using the polarization $h$, we consider the following stability conditions on $D^b(S)$.

\begin{dfn}\label{polarized stability}
Take real numbers $\alpha,\beta \in \RR$ such that $\alpha>0$ and $e^{\beta h+i\alpha h} \in \P^+_0(S)$. We define the stability condition $\sigma_{\alpha,\beta}=(Z_{\alpha.\beta}, \A_{\alpha,\beta})$ by $\sigma_{\alpha, \beta}:=\sigma_{\beta h, \alpha h}$ in Theorem \ref{stability K3}.
\end{dfn}

The stability conditions in Definition \ref{polarized stability} is fixed by polarized automorphisms in $\Aut(S,h)$.

\begin{prop}\label{fix K3}
Take real numbers $\alpha,\beta \in \RR$ such that $\alpha>0$ and $e^{\beta h+i\alpha h} \in \P^+_0(S)$.
For any automorphism $f \in \Aut(S,h)$, we have $f_*\sigma_{\alpha,\beta}=\sigma_{\alpha,\beta}$.
In particular, we have the natural inclusions 
\[\Aut(S,h) \hookrightarrow \Aut(D^b(S),\sigma_{\alpha,\beta}),  \]
and
\[\Aut_{\mathrm{s}}(S,h) \hookrightarrow \Aut_{\mathrm{s}}(D^b(S),\sigma_{\alpha,\beta}).\]
\end{prop}
\begin{proof}
Let $f \in \Aut(S,h)$ be a polarized automorphism of $S$. Note that $f_*h=h$.
Recall that $Z_{\alpha,\beta}(v)=(e^{\beta h+i\alpha h},v)$ for an object $v \in \widetilde{H}^{1,1}(S,\ZZ)$. We have 
\begin{eqnarray*}
Z_{\alpha, \beta}(f^*(v))
&=& (e^{\alpha h + i\beta h}, f^*(v))\\
&=& (f_*(e^{\alpha h+i \beta h}),v)\\
&=&  (e^{\alpha h+i \beta h}, v)\\
&=& Z_{\alpha,\beta}(v). 
\end{eqnarray*}
Moreover,  $f_*$ preserves the torsion pair on the abelian category $\mathrm{Coh}(S)$ in Definition \ref{tilt} and Theorem \ref{stability K3}. So $f_*$ preserves the heart $\A_{\alpha, \beta}$ of the bounded t-structure on $D^b(S)$.
\end{proof}

\section{Automorphisms of cubic fourfolds and stability conditons}
Section 5 is almost parallel to Section 4. In this section, we introduce certain subgroups of autoequivalence groups of Kuznetsov components of cubic fourfolds related to stability conditions. We will study automorphisms of cubic fourfolds and their relation with stability conditions.

\subsection{Subgroups of autoequivalence groups of Kuznetsov components}
In this subsection, we introduce subgroups of autoequivalence groups of Kuznetsov components as Subsection 4.1.

Let $X$ be a cubic fourfold. First, we compare the automorphism group $\Aut(X)$ with the autoequivalence group $\Aut(\D_X)$ of $\D_X$.
For an automorphism $f \in \Aut(X)$, we have the autoequivalence $f_* \in \Aut^{\mathrm{FM}}(\D_X)$ since $f_*\OO_X(k)$ for any integer $k \in \ZZ$. So we obtain the group homomorphism $\rho_1: \Aut(X) \to \Aut^{\mathrm{FM}}(\D_X)$.　
Note that the admissible subcategory $\D_X$ fits into the semi-orthogonal decomposition
\[ D^b(X)=\langle \OO_X(-1),\D_X,\OO_X,\OO_X(1) \rangle \]
by the Serre duality on $X$. We define the projection functor $\mathrm{pr}:D^b(X) \to \D_X$ by $\mathrm{pr}:=\mathbf{R}_{\OO_X(-1)}\mathbf{L}_{\OO_X}\mathbf{L}_{\OO_X(1)}$.
Recall the following property of  the projection functor $\mathrm{pr}:D^b(X) \to \D_X$. 

\begin{lem}[Proposition 1.4 in \cite{O}]\label{pr}
If $x \neq y \in X$, the object $\mathrm{pr}(\OO_x)$ is not isomorphic to $\mathrm{pr}(\OO_y)$.
\end{lem}
Lemma \ref{pr} implies the injectivity of $\rho_1$.
\begin{prop}\label{cubic injective}
The homomorphism $\rho_1$ is injective.
\end{prop}

\begin{proof}
Since $f_*\OO_X(k) \simeq \OO_X(k)$ for $k \in \ZZ$, for an automorphism $f \in \Aut(X)$, we have $f_* \circ \mathrm{pr} \simeq \mathrm{pr} \circ f_* $.
So we obtain $f_* \circ \mathrm{pr}(\OO_x) \simeq \mathrm{pr}(\OO_{f(x)})$ for any point $x \in X$. By Lemma \ref{pr}, the homomorphism $\rho_1$ is injective.
\end{proof}

We define the symplectic automorphism group of $X$ and the symplectic autoequivalence group of $\D_X$.

\begin{dfn}
An automorphism $f \in \Aut(X)$ is a symplectic automorphism of $X$ if $f^*$ acts on $H^{3,1}(X)$ trivially.
The group $\Aut_{\mathrm{s}}(X)$ of symplectic automorphisms of $X$ is called the symplectic automorphism group of $X$.
An autoequivalence $\Phi \in \Aut^{\mathrm{FM}}(\D_X)$ is a symplectic autoequivalence if $\Phi^H$ acts on $\widetilde{H}^{2,0}(\D_X)$ trivially. The group $\Aut^{\mathrm{FM}}_{\mathrm{s}}(\D_X)$ of symplectic autoequivalences of $\D_X$ is called the symplectic autoequivalence group of $\D_X$.
\end{dfn}

Symplectic automorphisms of cubic fourfolds are compatible with symplectic autoequivalences of Kuznetsov components. 

\begin{rem}
Let $f \in \Aut(X)$ be an automorphism of $X$. By Definition \ref{Mukai lattice}, $f \in \Aut_{\mathrm{s}}(X)$ if and only if $f_* \in \Aut^{\mathrm{FM}}(\D_X)$.
\end{rem}

As Definition \ref{stabilizer K3}, we introduce the following groups.

\begin{dfn}\label{stabilizer cubic}
For a stability condition $\sigma \in \Stab^*(\D_X)$, we define the group $\Aut(\D_X,\sigma)$ of autoequivalences fixing $\sigma$ by
\[\Aut(\D_X,\sigma):=\{\Phi \in \Aut^{\mathrm{FM}}(\D_X) \mid \Phi\sigma=\sigma\}.\]
Denote the intersection of $\Aut(\D_X,\sigma)$ and $\Aut^{\mathrm{FM}}_{\mathrm{s}}(\D_X)$ by $\Aut_{\mathrm{s}}(\D_X,\sigma)$.
\end{dfn}

As in Definition \ref{stabilizer Hodge K3}, we define the following groups in terms of Hodge theory.
\begin{dfn}\label{stabilizer Hodge cubic}
For a linear subspace $W \subset H^*(\D_X,\RR)$, we define 
\[\mathrm{O}_{\mathrm{Hodge}}(H^*(\D_X,\ZZ),W):=\{\varphi \in \mathrm{O}_{\mathrm{Hodge}}(H^*(\D_X,\ZZ)) \mid \varphi |_W=\mathrm{id}_W\}.\]
The weight two Hodge structure on $H^*(\D_X,\ZZ)$ defines the positive definite plane $P_X$ given by 
\[P_X:=\bigl(\widetilde{H}^{2,0}(\D_X)\oplus\widetilde{H}^{0,2}(\D_X)\bigr)\cap H^*(\D_X,\RR).  \]
For a stability condition $\sigma=(Z,\A) \in \Stab^*(\D_X)$, we have the positive definite plane $P_\sigma$ defined by
\[P_\sigma:=\langle \mathrm{Re}\bigl(\pi_X(\sigma)\bigl), \mathrm{Im}\bigl(\pi_X(\sigma)\bigr) \rangle_{\mathbb{R}}.\] 
For a stability condition $\sigma \in \Stab^*(\D_X)$, we put
\[ \Pi_\sigma:=P_X \oplus P_\sigma.\] 
\end{dfn}

For stability conditions in Theorem \ref{BLMS stability}, the corresponding positive definite planes come from $A_2$.
\begin{prop}[\cite{BLMS}, Proposition 9.11]\label{A2}
For a real number $0<\alpha<1/4$, we have $P_{\sigma_\alpha}=A_2 \otimes \RR$.
\end{prop}
\subsection{Automorphisms of cubic fourfolds}
In this subsection, we prove that automorphisms of cubic fourfolds fix stability conditions in Theorem \ref{BLMS stability}.
The goal of this subsection is the following proposition.
\begin{prop}\label{fix cubic}
Fix a real number $0<\alpha<1/4$. For an automorphism $f \in \Aut(X)$ of $X$, we have $f_*\sigma_\alpha=\sigma_\alpha$. In particular, we have the natural inclusions 
\[\Aut(X) \hookrightarrow \Aut(\D_X,\sigma_\alpha),  \]
and
\[\Aut_{\mathrm{s}}(X) \hookrightarrow \Aut_{\mathrm{s}}(\D_X,\sigma_\alpha).\]
\end{prop}
In Section 7, we will see the above inclusions are isomorphisms.
To prove Proposition \ref{fix cubic}, we study relations between automorphisms of cubic fourfolds and sheaves of Clifford algebras. 
Fix a line $l$ on $X$.
An automorphism $f \in \Aut(X)$ of $X$ induces $f_l: \mathrm{Bl}_lX \iso \mathrm{Bl}_{f(l)}X$ and $\tilde{f}_l \in \Aut(\mathbb{P}^3)$ such that we have the following commutative diagrams.

\[\xymatrix{
&\mathrm{Bl}_lX \ar[r]^{f_l} \ar[d]_{p_l} \ar@{}[dr]^{} & \mathrm{Bl}_{f(l)}X \ar[d]^{p_{f(l)}}&&&\mathrm{Bl}_lX \ar[r]^{f_l} \ar[d]_{\pi_l} & \mathrm{Bl}_{f(l)}X \ar[d]^{\pi_{f(l)}}\\
&X \ar[r]^{f} & X && &\mathbb{P}^3 \ar[r]^{\tilde{f}_l} & \mathbb{P}^3
}\]

The following is the relation between automorphisms of $X$ and the sheaves of Clifford algebras on $\mathbb{P}^3$.
\begin{lem}[Lemma 3.2 in \cite{Kuz08}, Lemma 7.2 in \cite{BLMS}]\label{Clifford}
Let $k \in \ZZ$ be an integer. For an automorphism $f \in \Aut(X)$ of $X$, we have $\tilde{f}_{l*}\B^l_k \simeq \B^{f(l)}_k$.
\end{lem}
\begin{proof}
In the proof of Lemma 3.2 in \cite{Kuz08}, this is observed in the more general situation. (Cf. Lemma 7.2 in \cite{BLMS}).
\end{proof}

The fully faithful functor in Subsection 3.4 is compatible with automorphisms of $X$.

\begin{prop}\label{automorphisms}
Let $f \in \Aut(X)$ be an automorphism of $X$. Then we have the following commutative diagram.
\[\xymatrix{\D_X \ar[d]_{\Psi_l \circ \mathbf{L}p^*_l} \ar[r]^{f_*} & \D_X \ar[d]^{\Psi_{f(l)} \circ \mathbf{L}p^*_{f(l)}} \\
D^b(\mathbb{P}^3, \B^l_0) \ar[r]^{\tilde{f}_{l*}} & D^b(\mathbb{P}^3,\B^{f(l)}_0)}  \]
\end{prop}
\begin{proof}
By Lemma \ref{Clifford}, we have the equivalence $\tilde{f}_{l*}: D^b(\mathbb{P}^3, \B^l_0) \iso D^b(\mathbb{P}^3, \B^{f(l)}_0)$. By the definition of $f_l$ and $\tilde{f}_l$, Lemma \ref{Clifford} and the exact sequence (A), we have $f_{l*}\E_l \simeq \E_{f(l)}$. So we obtain the desired commutative diagram.
\end{proof}

We prove Proposition \ref{fix cubic}.

\begin{proof}[Proof of Proposition \ref{fix cubic}]
Let $f \in \Aut(X)$ be an automorphism of $X$. Note that $f^*$ acts on $A_2$ trivially. By Proposition \ref{A2}, we have $Z_\alpha \circ f^*=Z_\alpha$. By Lemma \ref{Clifford} and Proposition \ref{automorphisms}, the equivalence $\tilde{f}_{l*}: D^b(\PP^3,\B^l_0) \iso D^b(\PP^3,\B^{f(l)}_0)$ induces the equivalence \[\tilde{f}_{l*}: \Psi_l \circ \mathbf{L}p^*_l(\D_X) \iso \Psi_{f(l)} \circ \mathbf{L}^*_{f(l)}(\D_X).\] Since $\tilde{f}_{l*}$ is compatible with tilting (cf. Definition \ref{tilt}) in the construction of $\A^l_\alpha$, we have $\tilde{f}_{l*}\A^l_\alpha=\A^{f(l)}_\alpha$. 
So we obtain $f_*\sigma^l_\alpha=\sigma^{f(l)}_\alpha$. By Theorem \ref{BLMS stability}, we have $f_*\sigma_\alpha=\sigma_\alpha$.
\end{proof}

\section{Automorphism groups of cubic fourfolds and Kuznetsov components}
In this section, we characterize automorphism groups of cubic fourfolds as subgroups of autoequivalence groups of Kuznetsov components.

\subsection{Automorphisms of cubic fourfolds as autoequivalences}
In this subsection, we give the statement of the theorem which is main in this section.
Let $X$ be a cubic fourfold. Fix a real number $0<\alpha<1/4$. Put $\sigma:=\sigma_\alpha \in \Stab^*(\D_X)$. By Proposition \ref{cubic injective} and Proposition \ref{fix cubic}, the  group homomorphism $\rho_1: \Aut(X) \to \Aut(\D_X, \sigma)$ is injective. We define the group homomorphism $\rho_2: \Aut(\D_X,\sigma) \to \mathrm{O}_{\mathrm{Hodge}}(H^*(\D_X,\ZZ), P_\sigma)$ by
 $\rho_2(\Phi):=\Phi^H$.
 The goal of this section is the following theorem.

\begin{thm}\label{K3 sigma model}
The group homomorphisms \[\rho_1: \Aut(X) \to \Aut(\D_X, \sigma),\]
 \[\rho_2:  \Aut(\D_X,\sigma) \to \mathrm{O}_{\mathrm{Hodge}}(H^*(\D_X,\ZZ), P_\sigma)\]
are isomorphisms. In particular, we have isomorphisms 
\[\rho_1: \Aut_{\mathrm{s}}(X) \to \Aut_{\mathrm{s}}(\D_X, \sigma),\]
 \[\rho_2:  \Aut_{\mathrm{s}}(\D_X,\sigma) \to \mathrm{O}_{\mathrm{Hodge}}(H^*(\D_X,\ZZ), \Pi_\sigma).\]
 \end{thm}
By Proposition \ref{cubic injective}, it is enough to show that $\rho_2 \circ \rho_1$ is surjective and $\rho_2$ is injective. 

First, we prove the surjectivity of $\rho_2 \circ \rho_1$.

\begin{prop}\label{surj}
The composition $\rho_2 \circ \rho_1: \Aut(X) \to  \mathrm{O}_{\mathrm{Hodge}}(H^*(\D_X,\ZZ), P_\sigma)$ of $\rho_1$ and $\rho_2$ is surjective.
\end{prop}
\begin{proof}
By Proposition \ref{A2}, we have $\mathrm{O}_{\mathrm{Hodge}}(H^*(\D_X,\ZZ), P_\sigma)=\mathrm{O}_{\mathrm{Hodge}}(H^*(\D_X,\ZZ), A_2 \otimes \RR)$. By Proposition \ref{prim}, 
there is the isomorphism 
\[v: \mathrm{O}_{\mathrm{Hodge}}(H^*(\D_X,\ZZ), A_2 \otimes \RR) \xrightarrow{\sim} \mathrm{O}_{\mathrm{Hodge}}(H^*_{\mathrm{prim}}(X,\ZZ)).\] 
Under this isomorphism, the natural homomorphism $\Aut(X) \to \mathrm{O}_{\mathrm{Hodge}}(H^4_{\mathrm{prim}}(X,\ZZ))$ is compatible with $\rho_2 \circ \rho_1$. 
By the Torelli theorem for cubic fourfolds, we obtain the statement.
\end{proof}
We will prove the following proposition.

\begin{prop}\label{inj2}
The homomorphism $\rho_2 : \Aut(\D_X, \sigma) \to  \mathrm{O}_{\mathrm{Hodge}}(H^*(\D_X,\ZZ), P_\sigma)$ is injective. Equivalently, if $\Phi \in \Aut^{\FM}(\D_X)$ satisfies $\Phi\sigma=\sigma$ and $\Phi^H=\mathrm{id}$, we have $\Phi=\mathrm{id}_{\D_X}$ in $\Aut(\D_X,\sigma)$.
\end{prop}

To prove Proposition \ref{inj2}, we study the Fano scheme $F(X)$ of lines on $X$.

\subsection{Fano schemes of lines}
In this subsection, we recall the description of Fano schemes of lines on cubic fourfolds in terms of moduli spaces of stable objects in Kuznetsov components. 

Let $X$ be a cubic fourfold.  The Fano scheme $F(X)$ of lines on $X$ can be described by the moduli theory on the Kuznetsov component $\D_X$. 
For a line $l$ on $X$, there is the exact sequence
\[    0 \to F_l \to \OO_X^{\oplus4} \to I_l(1) \to 0,\]
where $I_l$ is the ideal sheaf of $l$ in $X$.
Then we have $i^*(\O_l(1)) \simeq F_l[2]$. 
(Lemma 5.1 in \cite{KM})
We consider the autoequivalence $\Xi:= \mathbf{R}_{\mathcal{O}_X(-1)}(-\otimes \O_X(-1)) \in \Aut^{\mathrm{FM}}(\D_X)$ on $\D_X$. 
For a line $l \in F(X)$, we define the object $E_l \in \D_X$ as $E_l:=\Xi(F_l)$. For a line $l \in F(X)$, the object $E_l$ fits into the exact triangle $\OO_X(-1)[1] \to E_l \to I_l$ (\cite{MS}, Section 2.3). Li, pertusi and Zhao proved the following theorem.

\begin{thm}[\cite{LPZ}]\label{Fano}
Take a real number $0<\alpha<1/4$. Then we have an isomorphism $u: F(X) \iso M_{\sigma_\alpha}(v)$ given by 
$u([l]):=E_l$, where $v:=\lambda_1+\lambda_2$.
\end{thm}

Denote the Fano scheme of lines on $X$ by $F(X)$. For the universal line 
\[ \F(X):=\{ (x, l) \in X \times F(X) \mid x \in l\}\]
 on $X$, consider the natural projections $p: \F(X) \to F(X)$ and $q: \F(X) \to X$. 
We define an exact functor $\Sigma: \D_X \to D^b(F(X))$ as $\Sigma:=\mathbf{R}p_* \circ \mathbf{L}q^* \circ i$. Denote the right adjoint functor of $\Sigma$ by $\Sigma_R: D^b(F(X)) \to \D_X$.  
Addington proved the following proposition.
\begin{prop}[Section 5.1 in \cite{Add}]\label{for universal family}
For $[l] \in F(X)$, we have $\Sigma_R(\mathcal{O}_l(1)) \simeq F_l[1]$ 
\end{prop}

We modify the adjoint functors $\Sigma$ and $\Sigma_R$ for this situation.

\begin{dfn}
Take a real number $0<\alpha<1/4$ and put $\sigma:=\sigma_\alpha$.  We define the exact functors $P^*: \D_X \to D^b(M_\sigma(v))$ and $P_*: D^b(M_\sigma(v)) \to \D_X$ by 
\[P^*:=u_* \circ \Sigma \circ \Xi^{-1} \circ [1]: \D_X \to D^b(M_\sigma(v)),  \]
\[P_*:=[-1] \circ \Xi \circ \Sigma_R \circ u^*: D^b(M_\sigma(v)) \to \D_X. \]
Then $P_*$ is the right adjoint functor of $P^*$.
\end{dfn}

The above functors are related to an universal family of the moduli space $M_\sigma(v)$.

\begin{rem}\label{univP}
Since $\Xi \in \Aut^{\mathrm{FM}}(\D_X)$, the composition $i \circ P_*: D^b(M_\sigma(v)) \to D^b(X)$ is a Fourier-Mukai functor, that is, there is an object $\U \in D^b(M_\sigma(v) \times X)$ such that $i \circ P_* \simeq \Phi_\U$. By Proposition \ref{for universal family}, the object $\U$ is an universal family of the moduli space $M_\sigma(v)$ over $X$.
\end{rem}

We will need the following result about kernels of actions of automorphism groups on cohomology groups for irreducible holomorphic symplectic manifolds.

\begin{thm}[Theorem 2.1 in \cite{Bea}, Proposition 10 in \cite{HT}]\label{ker}
Take a positive integer $n>0$.  Let $M$ be an irreducible holomorphic symplectic manifold deformation equivalent to Hilbert scheme of $n$-points on K3 surfaces. Consider the group homomorphism 
\[\rho: \Aut(M) \to \mathrm{O}(H^2(M,\ZZ)), f \mapsto f^*.\]
Then we have $\mathrm{Ker}(\rho)=1$.
\end{thm}

Let $\Sigma^H: H^*(\D_X, \QQ) \to H^*(F(X),\QQ)$ be the linear map induced by $\Sigma: \D_X \to D^b(F(X))$.  

\begin{prop}\label{AJ}
The restriction $\Sigma^H|_{A_2^\perp} A_2^\perp \iso H^2_{\mathrm{prim}}(F(X),\ZZ)$ is a Hodge isometry, where $H^2_{\mathrm{prim}}(F(X),\ZZ)$ is the primitive cohomology group of $F(X)$ with respect to the Pl\"ucker polarization of $F(X)$. 
\end{prop}
\begin{proof}
The Abel-Jacobi map $p_*q^*: H^4_{\mathrm{prim}}(X,\ZZ)(-1) \iso H^2_{\mathrm{prim}}(F(X),\ZZ)$ is the Hodge isometry. By Proposition \ref{prim} and the definition of $\Sigma$, we have $\Sigma^H_{{A_2}^\perp}=p_*q^*v$.
\end{proof}

\begin{lem}
For an autoequivalence $\Phi \in \mathrm{Ker}(\rho_2)$, we have isomorphisms $\Phi \circ P_* \simeq P_*$ and $\P^* \circ \Phi \simeq P^*$.
\end{lem}
\begin{proof}
Let $\Phi  \in \Aut^{\FM}(\D_X)$ be an autoequivalence satisfying $\Phi\sigma=\sigma$ and $\Phi^H=\mathrm{id}$.
Then $\Phi$ induces an automorphism $\Phi_{\sigma,v}: M_\sigma(v) \iso M_\sigma(v), [E] \mapsto [\Phi(E)]$ such that the following diagram commutes.
\[\xymatrix{v^{\perp} \ar[d]_{\theta_{\sigma,v}} \ar[r]^{\Phi^H} & v^\perp  \ar[d]_{\theta_{\sigma, v}} \\
H^2(M_\sigma(v),\mathbb{Z}) \ar[r]^{{\Phi_{\sigma,v}}_*} & H^2(M_{\sigma}(v),\mathbb{Z})}  \]
Since $\Phi^H=\mathrm{id}$, we have ${\Phi_{\sigma,v}}_*=\mathrm{id}$. By Theorem \ref{ker}, we obtain $\Phi_{\sigma,v}=\mathrm{id}_{M_\sigma(v)}$.  
Since $\Phi$ is of Fourier-Mukai type, there is an object $\E \in D^b(X \times X)$ such that $\Phi \simeq i \circ \Phi_\E \circ i^*$.
By Remark \ref{univP}, we have $\Phi_\U \simeq i \circ P_*$, where $\U \in D^b(M_\sigma(v) \times X)$ is an universal family of $M_\sigma(v)$.
For $[E] \in M_\sigma(v)$, we have
\begin{eqnarray*}
\Phi_\E \circ \Phi_\U(\O_{[E]})
&\simeq&i^* \circ \Phi \circ i(E)\\
&\simeq&E. 
\end{eqnarray*}
The convolution product $\E \circ \U$, that is a Fourier-Mukai kernel of the composition $\Phi_\E \circ \Phi_\U$, is also an universal family of $M_\sigma(v)$. So there is a line bundle $L \in \mathrm{Pic}(M_\sigma(v))$ such that $\E \circ \U \simeq \U \otimes p_M^*L$, where $p_M: M_\sigma(v) \times X \to M_\sigma(v)$ is the projection. Since $i^* \circ i \simeq \mathrm{id}_{\D_X}$ and $\Phi_\E \circ \Phi_\U \simeq \Phi_\U \circ (- \otimes L)$, we obtain $\Phi \circ P_* \simeq P_* \circ (-\otimes L)$. By the uniqueness of left adjoint functors, we have $P^* \circ \Phi^{-1} \simeq (- \otimes L^{-1}) \circ P^*$ and it induces $(-\otimes L) \circ P^* \simeq P^* \circ \Phi$. Let $(P^*)^H: H^*(\D_X,\mathbb{Q}) \to H^*(M_\sigma(v), \mathbb{Q})$ be the linear map induced by $P^*:\D_X \to D^b(M_\sigma(v))$. Then we have $e^L \circ (P^*)^H=(P^*)^H$.
Take a non-zero divisor $D \in \mathrm{NS}(M_\sigma(v))$. By Proposition \ref{AJ}, there is an element $w \in \widetilde{H}^{1,1}(\D_X,\ZZ)$ such that $D=(P^*)^H(w)$. 
Then we have
\begin{eqnarray*}
D
&=& (P^*)^H(w)\\
&=&e^L \cdot  (P^*)^H(w)\\
&=& D+L\cdot D+\frac{1}{2}L^2 \cdot D+\frac{1}{6}L^3 \cdot D  . 
\end{eqnarray*}
By the Verbitsky's result (\cite{Ver}, Theorem 1.3), the natural map $\mathrm{Sym}^2(H^2(M_\sigma(v),\CC)) \to H^4(M_\sigma(v),\CC)$ is injective. So we have $L=0$. Therefore, we have $\Phi \circ P_* \simeq P_*$ and $P^* \circ \Phi \simeq P^*$. 
\end{proof}

\subsection{Comonads}
In this subsection,  we collect definitions and basic properties of comodules over comonads following \cite{elagin1}. Let $\mathcal{C}$ be a category.

\begin{dfn}
A {\it comonad} $\mathbb{T}=(T,\varepsilon,\delta)$ on  $\mathcal{C}$ consists of an endo-functor $T:\mathcal{C}\rightarrow\mathcal{C}$ and morphisms $\varepsilon:T\rightarrow {\rm id}_{\mathcal{C}}$ and $\delta:T\rightarrow T^2$ of functors such that the following diagrams commute.

\[\xymatrix{
&T \ar[r]^{\delta} \ar[d]_{\delta} \ar@{=}[dr]^{{\rm id}_T} & T^2 \ar[d]^{T\varepsilon}&&&T \ar[r]^{\delta} \ar[d]_{\delta} & T^2 \ar[d]^{T\delta}\\
&T^2 \ar[r]^{\varepsilon T} & T && &T^2 \ar[r]^{\delta T} & T^3
}\]

\end{dfn}

From adjoint functors, we can construct comonads. 
\begin{eg}\label{adj}
Let $P=(P^*\dashv P_*)$ be adjoint  functors $P^*:\mathcal{C}'\rightarrow\mathcal{C}$ and $P_*:\mathcal{C}\rightarrow\mathcal{C}'$. Denote the unit and the counit by $\eta_{P} :{\rm id}_{\mathcal{C}'}\rightarrow P_*\circ P^*$ and $\varepsilon_{P}:P^* \circ P_*\rightarrow {\rm id}_{\mathcal{C}}$ respectively.  We have an endo-functor $T_P:=P^*\circ P_*$ and morphisms $\delta_P:=P^*\eta_P P_*$ of functors. Then the triple $\mathbb{T}(P):=(T_P,\varepsilon_{P},\delta_P)$ is a comonad on $\mathcal{C}$.
\end{eg}

For a comonad, we have the notion of comodules over the comonad.

\begin{dfn} \label{comodule}
Let $\mathbb{T}=(T,\varepsilon,\delta)$ be a comonad on $\mathcal{C}$. A {\it comodule} over $\mathbb{T}$ is a pair $(C,\theta_C)$ of an object $C\in\mathcal{C}$ and a morphism $\theta_{C}:C\rightarrow T(C)$ such that 
\begin{itemize}
\item[(1)] $\varepsilon(C)\circ\theta_C={\rm id}_C$, and

\item[(2)]  the following diagram commutes.

$$\begin{CD}
C@>{\theta_C}>>T(C)\\
@V{\theta_C}VV @VVT({\theta_C})V\\
T(C)@>{\delta(C)}>>T^2(C).
\end{CD}$$
\end{itemize}
\end{dfn}

\begin{dfn}
Let $\mathbb{T}=(T,\varepsilon,\delta)$ be a comonad on $\mathcal{C}$.  We define the category $\mathcal{C}_{\mathbb{T}}$ of comodules over $\mathbb{T}$ as follow.
\begin{itemize}
\item[(1)] The set $\mathrm{Ob}(\mathcal{C}_{\mathbb{T}})$ of objects in $\mathcal{C}_{\mathbb{T}}$ consists of comodules over $\mathbb{T}$.

\item[(2)] For comodules $(C_1, \theta_{C_1}),(C_2, \theta_{C_2}) \in \mathrm{Ob}(\mathcal{C}_{\mathbb{T}})$,
\[{\rm Hom}_{\C_{\mathbb{T}}}((C_1,\theta_{C_1}),(C_2,\theta_{C_2})):=\{f\in\Hom_{\C}(C_1,C_2) \mid  T(f)\circ\theta_{C_1}=\theta_{C_2}\circ f \}.\]
\end{itemize}
\end{dfn}

We have the following natural functors.

\begin{dfn}
Let $\mathbb{T}=(T,\varepsilon, \delta)$ be a comonad on $\C$. We define a functor $Q_*: \C \to \C_{\mathbb{T}}$ as follow.
\begin{itemize}
\item[(1)]For an object $C \in \C$, set $Q_*(C):=(T(C), \delta(C))$. 
\item[(2)] For a morphism $f \in \Hom_{\C}(C_1,C_2)$, set $Q_*(f):=T(f)$.
\end{itemize}
We define a functor $Q^*:\C_{\mathbb{T}} \to \C$ as the forgetful functor. Then we have adjoint functors $Q=(Q^*\dashv Q_*)$. 
\end{dfn}

\begin{thm}[3.2.3 in \cite{BW}]\label{comparison functor}
Let $P=(P^*\dashv P_*)$ be adjoint  functors $P^*:\mathcal{C}\rightarrow\mathcal{C}'$ and $P_*:\mathcal{C}'\rightarrow\mathcal{C}$. Then there exists a functor $\Gamma_P: \C' \to \C_{\mathbb{T}(P)}$ unique up to an isomorphism such that $\Gamma_P \circ P_* \simeq Q_*$ and $Q^* \circ \Gamma_P \simeq P^*$. The functor $\Gamma_P: \C' \to \C_{\mathbb{T}(P)}$ is called the comparison functor.
\end{thm}

The following proposition gives  sufficient conditions for a comparison functor to be an equivalence.

\begin{prop}[{\cite[Theorem 3.9, Corollary 3.11]{elagin1}}]\label{comparison theorem}
Let $P=(P^*\dashv P_*)$ be adjoint  functors $P^*:\mathcal{C}\rightarrow\mathcal{C}'$ and $P_*:\mathcal{C}'\rightarrow\mathcal{C}$. If $\mathcal{C}$ is idempotent complete and the functor morphism $\eta_P :{\rm id}_{\mathcal{C}}\rightarrow P_* \circ P^*$ is a split mono, i.e. there exists a functor morphism $\zeta:P_* \circ P^*\rightarrow {\rm id}_{\mathcal{C}}$ such that $\zeta\circ\eta={\rm id}$, then $\Gamma_{P}:\mathcal{C}\rightarrow\mathcal{D}_{\mathbb{T}(P)}$ is an equivalence.
\end{prop}

\subsection{Application of comonads}
In this subsection, we prove that $\rho_2$ is injective as an application of comonads. 
Let $X$ be a cubic fourfold. Fix a real number $0< \alpha < 1/4$ and put $v:=[E_l]$, where $l$ is a line on $X$.  
Addington \cite{Add} proved the following theorem.

\begin{thm}[Section 5.1 in \cite{Add}]\label{RF}
The unit $\eta: \mathrm{id}_{\D_X} \to \Sigma_R \circ \Sigma$ is a split mono and $\Sigma_R \circ \Sigma \simeq \mathrm{id}_{\D_X} \oplus [-2]$.
\end{thm} 

As a consequence of Theorem \ref{RF}, we have the following statement.

\begin{cor}\label{split mono}
The unit $\eta: \mathrm{id}_{\D_X} \to P_* \circ P^*$ is a split mono and $P_* \circ P^* \simeq \mathrm{id}_{\D_X} \oplus [-2]$.
\end{cor}
We prove the Proposition \ref{inj2}.

\begin{proof}[Proof of Proposition \ref{inj2}]
As in Example \ref{adj}, let $\mathbb{T}(P)$ be the comonad on $D^b(M_\sigma(v))$ determined by the adjoint pair $P=(P^* \dashv P_*)$. 
By Theorem \ref{comparison functor}, there is a comparison functor $\Gamma_P: \D_X \to D^b(M_\sigma(v))_{\mathbb{T}(P)}$ unique up to an  isomorphism such that $\Gamma_P \circ P_* \simeq Q_*$ and $Q^* \circ \Gamma_P \simeq P^*$.   Since  $\Phi \circ P_* \simeq P_*$ and $P^* \circ \Phi \simeq P^*$, we have  $\Gamma_P  \circ \Phi \circ P_* \simeq Q_*$ and $Q^* \circ \Gamma_P \circ \Phi \simeq P^*$ So the composition $\Gamma_P \circ \Phi$ is also a comparison functor. Due to the uniqueness of comparison functors, we have $\Gamma_P \simeq \Gamma_P \circ \Phi$. By Corollary \ref{split mono} and Theorem \ref{comparison theorem}, the comparison functor $\Gamma_P$ is an equivalence. Therefore, we obtain $\Phi \simeq \mathrm{id}_{\D_X}$.

\end{proof}

As a corollary of Proposition \ref{inj2}, we have 
\begin{cor}
We define 
\[\Aut^0(\D_X):=\{\Phi \in \Aut^{\FM}(\D_X) \mid \Phi^H=\mathrm{id}, \Phi(\Stab^*(\D_X)) \subset \Stab^*(\D_X)\}.\]
Then the natural homomorphism $\Aut^0(\D_X) \to \mathrm{Deck}(\pi)$ is injective, where $\mathrm{Deck}(\pi)$ is the group of deck transformations of the covering $\pi$.
\end{cor}
\begin{proof}
Take $\Phi \in \mathrm{Ker}\bigl(\Aut^0(\D_X) \to \mathrm{Deck}(\pi)\bigr)$. Then $\Phi \sigma=\sigma$ holds. By Proposition \ref{inj2}, we have $\Phi=\mathrm{id}_{\D_X}$.
\end{proof}

\section{Automorphisms of cubic fourfolds and K3 surfaces}
In this section, we compare automorphisms of cubic fourfolds and autoequivalences of derived categories of K3 surfaces. For labeled automorphisms of cubic fourfolds, they induce polarized automorphisms of K3 surfaces. 
We introduce the notion of labeled automorphisms of cubic fourfolds.
\begin{dfn}
For a labeled cubic fourfold $(X,K)$, we define the labeled automorphism group $\Aut(X,K)$ of $(X,K)$ by 
\[\Aut(X,K):=\{f \in \Aut(X) \mid f^*|_K=\mathrm{id}_K\}.\] 
An automorphism $f \in \Aut(X)$ is called a labeled automorphism of $X$ if there is a rank two primitive sublattice $K \subset H^{2,2}(X,\ZZ)$ such that $(X, K)$ is a labeled cubic fourfold and $f \in \Aut(X,K)$.
\end{dfn}

We study polarized K3 surfaces associated to labeled cubic fourfolds from point of view of moduli spaces of stable objects in Kuznetsov components.

Take an integer $d$ that satisfies conditions $(*)$  and $(**)$. Let $(X,K)$ be a labeled cubic fourfold of discriminant $d$. By Remark \ref{label Mukai} and Theorem \ref{Hassett K3}, there is a polarized K3 surface $(S,h)$ of degree $d$ such that there is a Hodge isometry $\varphi: L^\perp_K \iso L^\perp_h$, where  $L_h$ is the sublatitce $H^0(S,\ZZ)\oplus \ZZ \cdot h \oplus H^4(S,\ZZ)$ of the Mukai lattice $H^*(S,\ZZ)$ of $S$. By Theorem 1.14.4 in \cite{Nik}, there is a Hodge isometry $\tilde{\varphi}: H^*(\D_X,\ZZ) \iso H^*(S,\ZZ)$ such that $\tilde{\varphi}|_{L^\perp_K}=\mathrm{id}_{L^\perp_K}$. Then $\tilde{\varphi}$ induces the isometry $\tilde{\varphi}|_{L_K}: L_K \iso L_h$. 
We will reconstruct the polarized K3 surface $(S,h)$ in terms of moduli spaces of stable objects in $\D_X$. Then we will obtain the nice equivalence $\D^b(S) \iso \D_X$ in the context of both Hodge theory and moduli theory.
First, we choose a Mukai vector and a stability condition on $\D_X$ to specify the moduli space. We put $v:=\tilde{\varphi}^{-1}(0,0,1) \in \widetilde{H}^{1,1}(\D_X,\ZZ)$. Note that the Mukai vector $v$ is a primitive isotropic.
Replacing $\tilde{\varphi}$ by $\bigl(\mathrm{id}_{H^0} \oplus -\mathrm{id}_{H^2} \oplus \mathrm{id}_{H^4}\bigr) \circ \tilde{\varphi}$ if necessary, we may assume that  $\Omega_1:=\tilde{\varphi}^{-1}(e^{ih})$ is contained in $\P^+_0(\D_X)$. Here, we use $d>6$ and Theorem \ref{stability K3} and Theorem \ref{BriK3}. We define $\Omega_n:=\tilde{\varphi}^{-1}(e^{inh}) \in \P^+_0(\D_X)$ for a positive integer $n$. 

\begin{dfn}
For a positive integer $n$, we define the group homomorphism
\[Z_n: \widetilde{H}^{1,1}(\D_X,\ZZ) \to \CC, w \mapsto (\Omega_n,w). \]
\end{dfn}

The following lemma will be used in the proof of Lemma \ref{generic}.

\begin{lem}[Lemma 8.2 in \cite{Bri08}]\label{Bri finite}
Fix $C>0$. For $\Omega \in \P^+_0(\D_X)$, there are finitely many elements $w \in \widetilde{H}^{1,1}(\D_X,\ZZ)$ such that $w^2 \geq -2$ and $|(\Omega,w)|\leq C$. 
\end{lem}

Recall that a stability condition $\sigma \in \Stab^*(\D_X)$ is $v$-generic if and only if a $\sigma$-semistable object $E \in \D_X$ with $v(E)=v$ is $\sigma$-stable by the primitivity of $v$.

\begin{lem}\label{generic}
For a positive integer $n$, let $\sigma_n=(Z_n, \A_n) \in \Stab^*(S)$ be a stability condition on $S$. 
Then there is a positive integer $N$ such that $\sigma_n$ is $v$-generic for any integer $n \geq N$.
\end{lem}
\begin{proof}
By Lemma \ref{Bri finite}, the set 
\[\Gamma:=\{w \in \widetilde{H}^{1,1}(\D_X,\ZZ) \mid w^2 \geq -2, |(\Omega_1,w)|\leq 1 \}\]
is a finite set.
For $w \in \widetilde{H}^{1,1}(\D_X,\ZZ)$, let $w_0$ be the $H^0$-part of  $\tilde{\varphi}(w)\in \widetilde{H}^{1,1}(S,\ZZ)$.
We define $N:=\mathrm{max}\{|w_0|\mid w \in \Gamma\}+1$. Take an integer $n \geq N$. Let $E$ be a $\sigma_n$-stable object with $v(E)=v$ and $\phi_{\sigma_n}(E)=1$.
By the definitions of $\Omega_n$ and $v$, we have $Z_n(v(E))=-1$.
Assume that there is a subobject $A$ of  $E$ in $\A_n$ such that $A$ is $\sigma_n$-stable and $\phi_{\sigma_n}(A)=1$. 
Denote $(r,c,m):=\tilde{\varphi}(v(A))$. Due to $\phi_{\sigma_n}(A)=1$, we have $h \cdot c=0$. Since $h$ is ample and $\mathrm{sign}(\mathrm{NS}(S))=(1,\rho(S))$, we have $c^2<-2$. 
If $r=0$, then $v(A)^2=c^2<-2$ holds and this is contradiction. So the integer $r$ is not zero.
Since $Z_n(A)$ and $Z_n(E/A)$ are negative real numbers and $Z_n(A)+Z_n(E/A)=-1$, we have $|Z_n(v(A))|\leq1$.
Denote $(r,c,m):=\tilde{\varphi}(v(A))$. 
Then we obtain
\begin{eqnarray*}
Z_{n}(v(A))
&=& (\Omega_n, v(A))\\
&=& (e^{inh},(r,c,m))\\
&=&  -m+\frac{1}{2}ndr\\
&=& (e^{ih},(rn,c,m))\\
&=& (\Omega_1,\tilde{\varphi}(rn,c,m)). 
\end{eqnarray*}
So $\tilde{\varphi}(rn,c,m) \in \Gamma$ holds. Note that $|rn| \geq N$ holds. By the definition of $N$, this is contradiction. 
Therefore, $E$ is a $\sigma_n$-stable object.
\end{proof}

We choose a $v$-generic stability condition $\sigma_n$ such that $\sigma_n$ is fixed by the action of $\Aut(X,K)$.

\begin{prop}\label{good heart}
Let $n$ be a positive integer. There is a stability condition $\sigma_n=(Z_n,\A_n) \in \Stab^*(\D_X)$ such that $f_*\sigma_n=\sigma_n$ holds for any automorphism $f \in \Aut(X,K)$.
\end{prop}
\begin{proof}
Recall that we have the isometry $\tilde{\varphi}|_{L_K}: L_K \iso L_h$.
Since $e^{inh} \in L_h\otimes \CC$, the class $\Omega_n$ is contained in  $L_K \otimes \CC$. 
Since $A_2 \subset L_K$, there are no $(-2)$ classes in $L^\perp_K$ by Theorem \ref{BLMS stability} and Proposition \ref{A2}. So the intersection 
$\P^+_0(\D_X) \cap (L_K \otimes \CC)$ is path-connected. Fix $0<\alpha<1/4$. Take a path $\gamma$ from $\pi_X(\sigma_\alpha)$ to $\Omega_n$.  
By Theorem \ref{BriK3}, the path $\gamma$ has the unique lift $\tilde{\gamma}: [0,1] \to \Stab^*(S)$ such that $\tilde{\gamma}(0)=\sigma_\alpha$. 
By the definition of $L_K$ and $\Aut(X,K)$, $f^*$ acts on the lattice $L_K$ trivially.
By Proposition \ref{fix cubic}, $f_* \circ \tilde{\gamma}$ is also a lift of $\gamma$ starting from $\sigma_\alpha$.  We define the stability condition $\sigma_n=(Z_n,\A_n) \in \Stab^*(\D_X)$ by $\sigma_n:=\tilde{\gamma}(1)$. By the uniqueness of lifts, we have $f_* \sigma_n=\sigma_n$ for any automorphism $f \in \Aut(X,K)$. 
\end{proof}
Fix a sufficiently large integer $n>0$ as in Lemma \ref{generic}. Take a stability condition $\sigma=(Z,\A):=\sigma_n$ in Proposition \ref{good heart}. 
Since $v$ is a primitive isotropic and $\sigma$ is $v$-generic, the moduli space $M_\sigma(v)$ of $\sigma$-stable objects with Mukai vector $v$ is a K3 surface. There is an isotropic $v' \in \widetilde{H}^{1,1}(\D_X,\ZZ)$ such that $(v,v')=-1$. Hence, $M_\sigma(v)$ is the fine moduli space. Take an universal family $\U \in D^b(M_\sigma(v) \times X)$ of $M_\sigma(v)$ over $X$. Then we have the equivalence $\Phi_\U: D^b(M_\sigma(v)) \iso \D_X$. By Theorem 1.3 in \cite{BM1}, we have the ample divisor $l_\sigma$ on $M_\sigma(v)$ such that $l_\sigma \cdot C=\mathrm{Im}Z(\Phi_\U(\mathcal{O}_C))$  holds for any curve $C$ on $M_\sigma(v)$.
We denote the ample divisor $l_\sigma/n$ by $\omega$.

\begin{prop}\label{isomorphism}
There are an isomorphism $t: M_\sigma(v) \iso S$ and the universal family $\U \in D^b(M_\sigma(v)\times X)$ of $M_\sigma(v)$ over $X$ such that $t^*h=\omega$ and $t_*=\tilde{\varphi}\circ \Phi^H_\U$.
\end{prop}
\begin{proof}
By the definition of the Mukai vector $v$, we have $\tilde{\varphi} \circ \Phi^H_\U(0,0,1)=(0,0,1)$. Since $(\tilde{\varphi} \circ \Phi^H_\U)^{-1}(1,0,0)$ is an isotropic and $((\tilde{\varphi} \circ \Phi^H_\U)^{-1}(1,0,0), (0,0,1))=-1$, there is a line bundle $L$ on $M_\sigma(v)$ such that $(\tilde{\varphi} \circ \Phi^H_\U)^{-1}(1,0,0)=(1,L,L^2/2)$. 
Equivalently, we have $\tilde{\varphi} \circ \Phi^H_\U(1,L,L^2/2)=(1,0,0)$. Let $p_M: M_\sigma(v) \times X \to M_\sigma(v)$ be the projection.  Replacing $\U$ by $\U \otimes p^*_M L^{-1}$, 
we may assume that $\tilde{\varphi} \circ \Phi^H_\U(1,0,0)=(1,0,0)$ holds. The restriction of $\tilde{\varphi} \circ \Phi^H_\U$ to the second cohomology group $H^2(M_\sigma(v),\ZZ)$ of $M_\sigma(v)$ is the Hodge isometry $\tilde{\varphi} \circ \Phi^H_\U|_{H^2(M_\sigma(v),\ZZ)}: H^2(M_\sigma(v),\ZZ) \iso H^2(S,\ZZ)$.
Take a curve $C$ on $M_\sigma(v)$.
We have
\begin{eqnarray*}
l_\sigma \cdot C
&=& Z(\Phi_\U(\mathcal{O}_C))\\
&=& (\mathrm{Im}(\Omega_n), v(\Phi_\U(\mathcal{O}_C)))\\
&=&  (\tilde{\varphi}^{-1}(0,nh,0),\Phi^H_\U(v(\mathcal{O}_C)))\\
&=& ((\tilde{\varphi} \circ \Phi^H_\U)^{-1}(0,nh,0), (0,C,\chi(\mathcal{O}_C))\\
&=& (\tilde{\varphi} \circ \Phi^H_\U)^{-1}(0,nh,0) \cdot C. 
\end{eqnarray*}
So we have $(\tilde{\varphi} \circ \Phi^H_\U)^{-1}(0,h,0)=(0,\omega,0)$, equivalently, $\tilde{\varphi} \circ \Phi^H_\U(0,\omega,0)=(0,h,0)$.
By the Torelli theorem for K3 surfaces, there is an isomorphism $t: M_\sigma(v) \iso S$ such that $t_*=\tilde{\varphi} \circ \Phi^H_\U$. In particular, we have $t^*h=\omega$.
\end{proof}
From now on, take the universal family $\U$ as in Proposition \ref{isomorphism}. By Proposition \ref{isomorphism}, we obtain the following.

\begin{rem}\label{sublattice L}
Let $L_\omega$ be the sublattice $H^0(M_\sigma(v),\ZZ) \oplus \ZZ \cdot \omega \oplus H^4(M_\sigma(v),\ZZ)$ of $H^*(\M_\sigma(v),\ZZ)$.
The Hodge isometry $\Phi^H_\U: H^*(M_\sigma(v),\ZZ) \iso H^*(\D_X,\ZZ)$ induces the isometry
\[ \Phi^H_\U|_{L_\omega}:L_\omega \iso L_K \]
and the Hodge isometry
\[ \Phi^H_\U|_{L^\perp_\omega}: L^\perp_\omega \iso L^\perp_K.\]

\end{rem}
The equivalence $\Phi_\U: D^b(M_\sigma(v)) \iso \D_X$ induces the isomorphism between the distinguished connected components of the spaces of stability conditions.

\begin{prop}\label{dist connected}
For a stability condition $\tau:=(W,\B) \in \Stab(\D_X)$, we define the stability condition $\Phi^*_\U\tau \in \Stab(M_\sigma(v))$ by
\[\Phi^*_\U\tau:=(W \circ \Phi^H_\U, \Phi^{-1}_\U(\B)). \]
Then we have the isomorphism
\[\Phi^*_\U: \Stab^*(\D_X) \iso \Stab^*(M_\sigma(v)).\]
\end{prop}
\begin{proof}
Let $\sigma$ be the stability condition as in Lemma \ref{generic}.
By the construction of $\sigma$ and Proposition \ref{isomorphism}, we have $Z\circ \Phi^H_\U(-)=(e^{in\omega},-)$.
By Lemma \ref{generic}, for any point $[E] \in M_\sigma(v)$, we have the $\sigma$-stable object $\Phi_\U(\mathcal{O}_{[E]}) \simeq E$.
By Proposition 10.3 in \cite{Bri08}, we have $\Phi^*_\U\sigma=\sigma_{0, n\omega}$ as in \ref{stability K3}.
So we obtain the isomorphism $\Phi^*_\U: \Stab^*(\D_X) \iso \Stab^*(M_\sigma(v))$.
\end{proof}
Using results in the previous section and this subsection, we obtain the following theorem.

\begin{thm}\label{group iso cubic K3}
There is a stability condition $\sigma_X \in \Stab^*(M_\sigma(v))$ such that we have the isomorphism 
\[(-)_\U: \Aut(X) \iso \Aut(D^b(M_\sigma(v)),\sigma_X), f \mapsto f_\U:=\Phi^{-1}_\U \circ f_* \circ \Phi_\U \] 
of groups.
Moreover, the restriction of $(-)_\U$ to the labeled automorphism group $\Aut(X,K)$ of $(X,K)$ induces the isomorphism 
\[(-)_\U : \Aut(X,K) \iso \Aut(M_\sigma(v), \omega)\]
of groups.
\end{thm}
\begin{proof}
Fix a real number $0<\alpha<1/4$. Put $\sigma_X:=\Phi^*_\U \sigma_\alpha \in \Stab(\D_X)$.
By Proposition \ref{dist connected}, the stability condition $\sigma_X$ is in the distinguished connected component $\Stab^*(M_\sigma(v))$.
By Proposition \ref{symmetry lattice} and Theorem \ref{K3 sigma model}, we have the isomorphism
 \[(-)_\E: \Aut(X) \iso \Aut(D^b(M_\sigma(v)),\sigma_X), f \mapsto f_\U:=\Phi^{-1}_\U \circ f_* \circ \Phi_\U. \] 
Let $f \in \Aut(X,K)$ be an automorphism. Due to Proposition \ref{good heart} and $v \in L_K$, for an automorphism $f \in \Aut(X,K)$, we obtain the automorphism 
\[f_{\sigma,v}: M_\sigma(v) \iso M_\sigma(v), [E] \mapsto [f_*E].\]  
Take a point $[E] \in M_\sigma(v)$. By Proposition \ref{good heart} and $v \in L_K$, we have \[\Phi^{-1}_\U \circ f_* \circ \Phi_\U(\mathcal{O}_{[E]}) \simeq \mathcal{O}_{[f_*E]}.\]
There is a line bundle $L$ on $M_\sigma(v)$ such that $\Phi_\U \circ f_{\sigma,v*} \circ (-\otimes L) \simeq f_* \circ \Phi_\U $. Since $\Phi^H_\U(1,0,0)= \tilde{\varphi}(1,0,0)$ in $L_K$,  we have $\Phi^H_\U(1,L,L^2/2)=\Phi^H_\U(1,0,0)$. So $L=0$ holds. Due to $\Phi^H_\U(0,\omega,0)=\tilde{\varphi}(0,h,0)$ in $L_K$, we obtain $f^*_{\sigma,v}\omega=\omega$.
Hence, we have $f_{\sigma,v*}=f_\U$. We shall prove that $(-)_\U: \Aut(X,K) \hookrightarrow \Aut(M_\sigma(v),\omega)$ is surjective.  
Take an automorphism $g \in \Aut(M_\sigma(v), \omega)$. By Proposition \ref{prim} and Remark \ref{sublattice L}, $\Phi^H_\U \circ g_* \circ (\Phi^{-1}_\U)^H$ induces the Hodge isometry $\Phi^H_\U \circ g_* \circ (\Phi^{-1}_\U)^H|_{A^\perp_2}: H^4_{\mathrm{prim}}(X,\ZZ) \iso H^4_{\mathrm{prim}}(X,\ZZ)$. By the Torelli theorem  for cubic fourfolds \cite{Voi}, there is an automorphism $f \in \Aut(X)$ such that 
$f_*=\Phi^H_\U \circ g_* \circ (\Phi^{-1}_\U)^H|_{A^\perp_2}$. By Remark \ref{sublattice L}, we have $f \in \Aut(X,K)$.  Then the equality $f_{\sigma,v*}=g_*$ holds. By the Torelli theorem for K3 surfaces, we obtain $f_{\sigma,v}=g$. Therefore, the homomorphism $(-)_{\U}: \Aut(X,K) \hookrightarrow \Aut(M_\sigma(v),\omega)$ is surjective.
\end{proof}

Therefore, we have obtained Theorem \ref{intromain2} from Proposition \ref{isomorphism}, Remark \ref{sublattice L} and Theorem \ref{group iso cubic K3}.

\section{Symplectic automorphisms of cubic fourfolds and associated K3 surfaces}
In this section, we study relations between symplectic automorphisms of cubic fourfolds and autoequivalences of derived categories of K3 surfaces.

Let $X$ be a cubic fourfold and put $G:=\Aut_{\mathrm{s}}(X)$. We define the coinvariant sublattice $S_G(X)$ of $H^4(X,\ZZ)$ by
$S_G(X):=(H^4(X,\ZZ)^G)^\perp$. Let $T_X$ be the orthogonal complement of $H^{2,2}(X,\ZZ)$ in $H^4(X,\ZZ)$. By Proposition \ref{prim}, we have $T_{\D_X} \simeq T_X(-1)$.

\begin{lem}
For a symplectic automorphism $f \in G$ of $X$, we have $f^*|_{T_X}=\mathrm{id}_{T_X}$.
\end{lem}
\begin{proof}
The proof is the same as the case of K3 surfaces. See \cite{HuyK3} (Remark 3.3 in Section 3 and Remark 1.2 in Section 15).
\end{proof}

\begin{lem}\label{pic vs coinv}
The inequality $\rho(\D_X) \geq \mathrm{rk}(S_G(X))$ holds.
\end{lem}
\begin{proof}

Since $H^2 \in H^4(X,\ZZ)^G$ and $T_X \subset H^4(X,\ZZ)^G$, we have $S_G(X) \subset H^{2,2}_{\mathrm{prim}}(X,\ZZ)$.
\end{proof}

Existence of non-trivial symplectic automorphisms of $X$ is the strong constraint for the Picard number of the Kuznetsov component.

\begin{prop}[\cite{LZ}]
If $G$ is not isomorphic to the trivial group $1$ or the cyclic group $\ZZ/2\ZZ$ of order two, then we have $\rho(\D_X) \geq 12$.
\end{prop}
\begin{proof}
Assume that $G$ is not isomorphic to the trivial group $1$ or the cyclic group $\ZZ/2\ZZ$ of order two.
By Theorem 1.2 in \cite{LZ}, we have $ \mathrm{rk}(S_G(X)) \geq 12$. By Lemma \ref{pic vs coinv}, we obtain the inequality $\rho(\D_X) \geq 12$.
\end{proof}

We note the following.

\begin{thm}\label{symplectic exist}
If $G$ is not isomorphic to the trivial group $1$ or $\ZZ/2\ZZ$, then there exists the unique K3 surface $S$ such that $\D_X \simeq D^b(S)$.
\end{thm}
\begin{proof}
By Corollary 2.10 in \cite{Mor}, there is the unique K3 surface $S$ such that $T_{\D_X}$ has the unique primitive embedding $T_{\D_X} \hookrightarrow H^2(S,\ZZ)$ up to isometries of $H^2(S,\ZZ)$ such that $T_{\D_X}=T_S$. Now, we have two primitive embeddings
\[ T_{\D_X} \hookrightarrow H^2(S,\ZZ) \hookrightarrow H^*(S,\ZZ),\]
\[T_{\D_X} \hookrightarrow H^*(\D_X,\ZZ).\]  
For the primitive embedding $T_{\D_X} \hookrightarrow H^2(S,\ZZ) \hookrightarrow H^*(S,\ZZ)$, we have $U \simeq H^0(S,\ZZ) \oplus H^4(S,\ZZ) \subset T^\perp_{\D_X}$.
By Remark 1.13 in \cite{HuyK3}, there is a Hodge isometry $H^*(S,\ZZ) \simeq H^*(\D_X,\ZZ)$. So we have $U \subset \widetilde{H}^{1,1}(\D_X,\ZZ)$. By Theorem \ref{derivedTorelli} and the uniqueness of $S$, we obtain $\D_X \simeq D^b(S)$.
\end{proof}

The following is the list of orders of finite symplectic automorphisms of K3 surfaces.

\begin{rem}[\cite{Nik2}]\label{order K3}
Let $S$ be a K3 surface. If $f$ is a symplectic automorphism of $S$ with the finite order, then we have 
$1 \leq \mathrm{ord}(f) \leq 8$.
\end{rem}

We can find examples of finite symplectic autoequivalences of K3 surfaces via symplectic automorphisms of cubic fourfolds. 
In some cases, they are not conjugate to symplectic automorphisms of K3 surfaces.

\begin{eg}\label{Fermat}
Let $X$ be the Fermat cubic fourfold, that is defined by the equation
\[x^3_1+x^3_2+x^3_3+x^3_4+x^3_5+x^3_6=0.\]
By Theorem 1.8 (1) in \cite{LZ}, the Fermat cubic fourfold $X$ is the unique cubic fourfold with a symplectic automorphism $f_9$ of order $9$. 
By Theorem \ref{symplectic exist}, there is the unique K3 surface $S$ such that $\D_X \simeq D^b(S)$. 
Then $f_9$ induces the symplectic autoequivalence $\Phi_9 \in \Aut_{\mathrm{s}}(D^b(S))$ of order $9$.  
By Remark \ref{order K3}, $\Phi_9$ is not conjugate to symplectic automorphisms of $S$.
\end{eg}

\begin{eg}\label{Klein}
Let $X$ be the Klein cubic fourfold, that is defined by the equation
\[x^3_1+x^2_2x_3+x^2_3x_4+x^2_4x_5+x^2_5x_6+x^2_6x_2=0.\]
This is the triple cover of $\mathbb{P}^4$ branched along the Klein cubic threefold in \cite{Adl}.
By Theorem 1.8 (5) in \cite{LZ}, the Klein cubic fourfold $X$ is the unique cubic fourfold with a symplectic automorphism $f_{11}$ of order $11$ and its symplectic automorphism group is the finite simple group $L_2(11)$ (cf. \cite{Adl}). 
By Theorem \ref{symplectic exist}, there is the unique K3 surface $S$ such that $\D_X \simeq D^b(S)$. 
Then $f_{11}$ induces the symplectic autoequivalence $\Phi_{11} \in \Aut_{\mathrm{s}}(D^b(S))$ of order $11$.  
By Remark \ref{order K3}, $\Phi_{11}$ is not conjugate to symplectic automorphisms of $S$.
\end{eg}

See Theorem 1.2 and 1.8 in \cite{LZ} for other interesting cubic fourfolds.

\end{document}